\documentclass[12pt]{amsart}
\usepackage[T1]{fontenc}
\usepackage[latin1]{inputenc}
\usepackage{typearea}
\usepackage{geometry}
\usepackage{ulem}
\usepackage{amsmath}
\usepackage{amssymb}
\usepackage{latexsym}
\usepackage{enumerate}
\usepackage{amsthm}
\usepackage[all]{xy}
\usepackage{hhline}
\usepackage{epsf} 
\usepackage{cite}
\usepackage{verbatim}
\usepackage{color}

\newcommand{\V}{\mathbb{V}}

\newtheorem{theorem}{{\sc Theorem}}[section]
\newtheorem{cor}[theorem]{{\sc Corollary}}

\newtheorem{lemma}[theorem]{{\sc Lemma}}
\newtheorem{prop}[theorem]{{\sc Proposition}}

\theoremstyle{remark}
\newtheorem{remark}[theorem]{{\sc Remark}}

\theoremstyle{definition}

\newcommand{\R}{\mathbb{R} }
\newcommand{\N}{\mathbb{N} }

\newcommand{\B}{\mathcal{B}}
\newcommand{\F}{\mathcal{F}}
\newcommand{\G}{\mathcal{G}}

\newcommand{\D}{\mathcal{D}}

\newcommand{\Tr}{\textnormal{Tr}}
\newcommand{\Prob}{\mathbb{P}}
\newcommand{\E}{\mathbb{E}}

\providecommand{\abs}[1]{\lvert #1\rvert}
\providecommand{\babs}[1]{\bigl\lvert #1\bigr\rvert}
\providecommand{\Babs}[1]{\Bigl\lvert #1\Bigr\rvert}

\providecommand{\fnorm}[1]{\lVert #1\rVert_\infty}

\providecommand{\Opnorm}[1]{\lVert #1\rVert_{\op}}
\providecommand{\Enorm}[1]{\lVert #1\rVert_2}
\providecommand{\HSnorm}[1]{\lVert #1\rVert_{\HS}}
\DeclareMathOperator{\diag}{diag}
\DeclareMathOperator{\Var}{Var}

\DeclareMathOperator{\Cov}{Cov}

\DeclareMathOperator{\op}{op}
\DeclareMathOperator{\Hess}{Hess}
\DeclareMathOperator{\HS}{H.S.}

\renewcommand{\phi}{\varphi}
\renewcommand{\epsilon}{\varepsilon}

\renewcommand{\rho}{\varrho}

\begin{document}
\title[Quantitative de Jong]{Quantitative de Jong theorems \\ in any dimension}
\author{Christian D\"obler \and Giovanni Peccati}
\thanks{\noindent Universit\'{e} du Luxembourg, Unit\'{e} de Recherche en Math\'{e}matiques \\
E-mails: christian.doebler@uni.lu, giovanni.peccati@uni.lu }
\begin{abstract}  We develop a new quantitative approach to a multidimensional version of the well-known {\it de Jong's central limit theorem} under optimal conditions, stating that a sequence of Hoeffding degenerate $U$-statistics whose fourth cumulants converge to zero satisfies a CLT, as soon as a Lindeberg-Feller type condition is verified. Our approach allows one to deduce explicit (and presumably optimal) Wasserstein bounds in the case of general $U$-statistics of 
arbitrary order $d\geq1$. One of our main findings is that, for vectors of $U$-statistics satisfying de Jong' s conditions and whose covariances admit a limit, componentwise convergence systematically implies joint convergence to Gaussian: this is the first instance in which such a phenomenon is described outside the frameworks of homogeneous chaoses and of diffusive Markov semigroups.
\end{abstract}

\maketitle

\section{Introduction, framework and main results}\label{intro}

\subsection{Overview}
Let $\{ W_n : n\geq 1\}$ be a sequence of unit variance $U$-statistics of order $d\geq 1$ (not necessarily symmetric) with underlying independent data $X_1,\dotsc,X_n$, that are degenerate in the sense of Hoeffding (see Section \ref{intro1} for formal definitions) and have a finite fourth moment. In the landmark paper \cite{deJo90} (see also \cite{deJo89}), P. de Jong proved the following remarkable fact, valid as $n\to \infty$: if $\E[W_n^4]\to 3$ and a Lindeberg-Feller-type condition is verified, then $W_n$ converges in distribution towards a standard Gaussian random variable $Z$ (note that $3 = \E [Z^4]$). This surprising result represents a drastic simplification of the method of moments and cumulants (see e.g. \cite[Section A.3]{NouPecbook}), which should be  contrasted with the `typical' non-central asymptotic behaviour of degenerate $U$-statistics of a fixed order $d\geq2$ and with a fixed kernel --- see e.g. \cite{Greg77}, \cite{ser-book}, \cite{RubVi80}, \cite{DynMan83} or \cite[Ch. 11]{J-book} 
; it also provides a general explanation of the ubiquitous emergence of the Gaussian distribution in geometric models where counting statistics can be naturally represented in terms of degenerate $U$-statistics, see e.g. \cite{JJ, RePe-book, Pen-book}.  

\medskip

One should notice that de Jong's central limit theorem (CLT) is a one dimensional {\it qualitative} statement: in particular, it does not provide any meaningful information about the rate of convergence of the law of $W_n$ towards the target Gaussian distribution. Our aim in this paper is to use {\it Stein's method of exchangeable pairs}, as originally developed in Stein's monograph \cite{St86}, in order to prove new quantitative and multidimensional versions of de Jong's central limit theorem under minimal conditions, in the setting of {\it degenerate and non-symmetric} $U$-statistics that do not necessarily have the form of homogeneous sums. In particular, we are interested in characterizing the joint convergence of those vectors of degenerate $U$-statistics, whose components verify one-dimensional CLTs.

\medskip

One of the main motivations for pursuing our goal is that the findings of \cite{deJo89} have anticipated a modern and very fruitful direction of research, where tools of infinite-dimensional calculus are used in order to deduce {\it fourth moment theorems} in the spirit of de Jong (but, crucially, without the use of Lindeberg-Feller-type conditions) for random variables belonging to {\it the homogeneous chaos} of some general random field. The best-known results in this area gravitate around the main discovery of \cite{NuaPec05} (as well as its multidimensional extension \cite{PeTu}), where it is proved that a sequence of normalized random variables $\{Y_n : n\geq 1\}$, belonging to a fixed {\it Wiener chaos} of a Gaussian field, verifies a central limit theorem (CLT) if and only if $\E[Y_n^4]\to 3$. The combined use of {\it Malliavin calculus} and {\it Stein's method} has consequently allowed one to deduce strong quantitative versions with explicit Berry-Esseen bounds of these results (see \cite{NouPec09a,NouPecbook}), and it is therefore a natural question to ask whether the original CLT by de Jong can be endowed with explicit bounds, that are comparable with those available in a Gaussian setting. 

\medskip

The reader can consult the constantly updated webpage
$$
\mbox{\tt https://sites.google.com/site/malliavinstein/home}
$$ 
for an overview of the  emerging domain of research connected to \cite{NouPec09a,NouPecbook, NuaPec05, PeTu}. Among the many notable ramifications of the results of \cite{NouPec09a, NuaPec05} to which our findings should be compared, we quote: \cite{KRT, NPR-ejp, PrTo} for results involving homogeneous sums in the {\it Rademacher} (also called {\it Walsh}) chaos, \cite{EiThae14, LRP1, LRP2, PSTU, PZ-ejp, ReSch, Sch-Poisson} for the analysis of Poissonized $U$-statistics living in the Wiener chaos associated with a Poisson measure, \cite{Ariz, BP-free, KNPS, NPSP} for fourth moment theorems involving homogeneous sums in a non-commutative setting, and \cite{ACP, CNPP, ledoux-aop} for results in the setting of chaotic random variables associated with a diffusive Markov semigroup. Central and non-central quantitative versions of de Jong's results in the case of fully symmetric Poissonized $U$-statistics can be found in \cite{EiThae14, FT, PTh}.

\medskip

Two sets of references are particularly relevant for the present work:

\medskip

\begin{enumerate}

\item[(a)] In reference \cite{NPR-aop} (see also \cite{PZ-ber}) de Jong's CLT in the special case of {\it homogeneous sums} was studied in the framework of the powerful theory of universality and influence functions initiated in \cite{MOO}. In particular, explicit 
bounds 
were obtained for vectors of homogeneous sums satisfying a CLT. 

\smallskip

\item[(b)] In the already quoted reference \cite{PeTu}, the following striking phenomenon was discovered. For $r\geq 2$, let $Y^{m} = (Y^{m}_1,...,Y^{m}_r)$, $m\geq 1$, be a sequence of random vectors whose components live in a fixed Wiener chaos, and assume that the covariance matrix of $Y^{m}$ converges to some ${\bf \Sigma} \geq 0$ and that each component $Y^m_i$ verifies a CLT; then, $Y^m$ converges in distribution towards a Gaussian vector with covariance ${\bf \Sigma}$, that is: {\it for vectors of random variables living in a fixed chaos, componentwise convergence to Gaussian, systematically implies joint convergence}. As explained e.g. in \cite[Chapter 6]{NouPecbook}, such a phenomenon serves as a key stepping stone in order to deduce Gaussian approximations for general functionals of Gaussian fields. Since then, this result has been extended (at least, partially), to the framework of the homogeneous chaos associated with a Poisson measure (see \cite{BP-geo, PZ-ejp}), to general vectors of homogeneous sums (see \cite[Section 7]{NPR-aop} and \cite{NPPSi}), to the free probability setting (see \cite{NPSP}), as well as to the framework of Markov chaoses (see \cite{CNPP}).

\end{enumerate}

\medskip

The achievements of the present paper are twofolds: 

\medskip

\begin{enumerate}

\item[(1)] On the one hand, we will obtain a general quantitative version of the one-dimensional de Jong CLT, displaying explicit bounds on the 1-Wasserstein distance. As anticipated, we will do that in the full general setting of degenerate $U$-statistics that do not necessarily have the form of homogeneous sums, and that are not necessarily symmetric. In particular, this extends the CLTs for homogeneous sums proved in \cite{NPR-aop}, as well as the results for Poissonized and symmetric $U$-statistics proved in \cite{EiThae14, LRP1}.

\medskip

\item[(2)] On the other hand, we will deduce (quantitative) multidimensional versions of de Jong theorems, showing that the crucial phenomenon observed in \cite{PeTu} (see the discussion at Point (b) above) basically extends to the framework of degenerate $U$-statistics. Our main theorems on the matter show that the case of $U$-statistics of the same order must take into account at least one cumulant  of order four --- thus echoing recent results from \cite{CNPP}. Our forthcoming Theorem \ref{t:mq} marks the first instance in which the phenomenon observed in \cite{PeTu} is described in full generality, outside the frameworks of homogeneous chaoses, and of the chaoses associated with a diffusive Markov semigroup.

\end{enumerate}

\medskip

We will now describe our setting and our main results in more detail. 

\subsection{Main results, I: univariate normal approximations}\label{intro1}
Let us fix the following setup and notation, which we essentially adopt from \cite{deJo90}.  We refer the reader to the classical references \cite{Hoeffding, KB-book, KR, ser-book, vitale}, as well as to the more recent works \cite{ElDP, ElDPP, LRP, P-aop}, for an introduction to degenerate $U$-statistics, Hoeffding decompositions and their use in stochastic analysis. 

\smallskip

 Let $(\Omega,\F,\Prob)$ be a probability space and for an integer $n\geq1$ let $X_1,\dotsc,X_n$ be independent random elements on this space assuming values in the respective 
measurable spaces $(E_1,\mathcal{E}_1),\dotsc,(E_n,\mathcal{E}_n)$. Further, assume that 
\begin{equation*}
 f:\prod_{j=1}^n E_j\rightarrow\R\quad\text{is}\quad\bigotimes_{j=1}^n\mathcal{E}_j-\B(\R)\text{ - measurable}
\end{equation*}
and that 
\begin{equation*}
W:=f(X_1,\dotsc,X_n) \in L^4(\Prob)
\end{equation*}
satisfies 
\begin{equation}\label{stw}
 \E[W]=0\quad\text{and}\quad \E[W^2]=1\,.
\end{equation}
We write 
\begin{equation*}
 [n]:=\{1,\dotsc,n\}
\end{equation*}
and for $J\subseteq[n]$ we also write
\begin{equation*}
 \F_J:=\sigma(X_j,\,j\in J)\,.
\end{equation*}
We write 
\begin{equation}\label{hoeffding}
 W=\sum_{J\subseteq[n]}W_J
\end{equation}
to indicate the {\it Hoeffding decomposition} of $W$. Note that this means that, for each $J\subseteq[n]$,  $W_J$ is $\F_J$-measurable and that 
\begin{equation}\label{hddef}
 \E[W_J\,|\,\F_K]=0\,,
\end{equation}
whenever $J\nsubseteq K$. It is well-known that $W$ admits a Hoeffding decomposition of the type \eqref{hoeffding}, as long as $W\in L^1(\Prob)$ and that 
it is almost surely unique and given by 
\begin{equation}\label{hdw}
W_J=\sum_{L\subseteq J}(-1)^{\abs{J}-\abs{L}}\E\bigl[W\,\bigl|\,\F_L\bigr]\,,\quad{J\subseteq[n]}\,.
\end{equation}

We can thus write 
\begin{equation}\label{kernelsfj}
 W_J=f_J(X_j,\,j\in J)
\end{equation}
for some measurable functions
\begin{equation*}
 f_J:\prod_{j\in J}E_j\rightarrow\R\,,\quad J\subseteq[n]\,.
\end{equation*}
Let us also define
\begin{equation*}
 \sigma_J^2:=\Var(W_J)\,,\quad J\subseteq[n]\,.
\end{equation*}

One major assumption in what follows will be that, for some fixed integer $d\in[n]$, $W$ is a \textit{ degenerate $U$-statistic of order $d$} (or $d$-degenerate $U$-statistic), i.e. that the Hoeffding decomposition \eqref{hoeffding} has the form 
\begin{equation}\label{dhom}
 W=\sum_{J\in\D_d}W_J\,,
\end{equation}
where 
\[\D_d:=\{J\subseteq[n]\,:\, \abs{J}=d\}   \] 
denotes the collection of all $\binom{n}{d}$ $d$-subsets of $[n]$.
Equivalently, we have $W_K=0$ whenever $K\subseteq[n]$ is such that $\abs{K}\not= d$. Hence, we have 
\begin{equation*}
 W=f(X_1,\dotsc,X_n)=\sum_{J\in\D_d} f_J(X_j,\,j\in J)\,.
 \end{equation*}

 The next lemma lists important properties of the Hoeffding decomposition of $W$ which will be used without further mention.
 \begin{lemma}\label{hdle}
Let the above notation and definitions prevail. Then, one has the following properties: 
 \begin{enumerate}[{\normalfont 1)}]
  \item Whenever $J,K_1,\dotsc,K_s\in\D_d$ are such that \label{(4)}
\begin{equation*}
 J\nsubseteq\bigcup_{i=1}^s K_i=:K
\end{equation*}
and $g(W_{K_1},\dotsc,W_{K_s})$ is square-integrable, then 
\begin{equation*}
 \E\bigl[W_Jg\bigl(W_{K_1},\dotsc,W_{K_s}\bigr)\bigr]=0\,.
\end{equation*}
In particular, $W_J$, $J\in\D_d$, are uncorrelated.
 
 \item For all $J,K\in\D_d$ such that $J\not=K$ we have $\E[W_J\,|\,\F_K]=0$.\label{(3)}
 \item For all $J,K\in\D_d$ we have $\E[W_J]=0$ and $\E[W_JW_K]=\delta_{J,K}\sigma_J^2$.\label{(1)}
 \item $\displaystyle\sum_{J\in\D_d}\sigma_J^2=1$.\label{(2)}

\item Whenever $J_1,\dotsc,J_r,K_1,\dotsc,K_s\in\D_d$ are such that \label{(5)}
 \begin{equation*}
  \Bigl(\bigcup_{l=1}^r J_l\Bigr)\cap \Bigl(\bigcup_{i=1}^s K_i\Bigr)=\emptyset\,,
 \end{equation*}
then the families $\{W_{J_l}\,:\,l=1,\dotsc,r\}$ and $\{W_{K_i}\,:\,i=1,\dotsc,s\}$ are independent, i.e. the summands $W_J$, $J\in\D_d$, are \textit{dissociated} as defined in {\rm \cite{GinSib}.}
\end{enumerate}
\end{lemma}

\begin{proof}
Point 1) is a consequence of the degeneracy property \eqref{hddef} because, 
\begin{equation*}
 \E[W_J\,|\,\F_K]=0\,,
\end{equation*}
as $J\not\subseteq K$ and, hence, by conditioning we have
\begin{align*}
 \E\bigl[W_Jg\bigl(W_{K_1},\dotsc,W_{K_s}\bigr)\bigr]&=\E\bigl[g\bigl(W_{K_1},\dotsc,W_{K_s}\bigr)\E[W_J\,|\,\F_K]\bigr]=0\,.
\end{align*}
Now, Point 2) follows since it is a special case of Point 1) and also Point 3) and Point 4) are immediately implied by Point 1), in view of assumption \eqref{stw}. Finally, Point 5) follows from independence as well as the disjoint block theorem.\\
\end{proof}

Let us furthermore define the quantity
\begin{equation}\label{rhodef}
 \rho^2:=\rho_n^2:=\max_{1\leq i\leq n}\sum_{\substack{K\in\D_d:\\ i\in K}}\sigma_K^2\,.
\end{equation} 

The next result corresponds to de Jong's celebrated (qualitative) CLT discussed in Section 1.1.

\begin{theorem}[See \cite{deJo90}]\label{t:ordj} Fix $d\geq 1$, and let $\{n_m : m\geq 1\}$ be a sequence of integers diverging to infinity. Let $\{W_m : m\geq 1 \}$ be a sequence of unit variance degenerate $U$-statistics of order $d$, such that each $W_m$ is a function of the vector of independent variables $(X_1^{(m)},...,X_{n_m}^{(m)})$. 
Then, as $m\to \infty$, if $\E[ W_m^4]\to 3$ and $\rho_{n_m}^2\to 0$, one has that $W_m$ converges in distribution towards a standard Gaussian random variable.
\end{theorem}

{Note that the condition $\lim_{m\to\infty}\rho_{n_m}^2=0$ guarantees that, as $m\to\infty$, the influence of each of the random variables $(X_1^{(m)},...,X_{n_m}^{(m)})$ on the total variance of $W_{n_m}$ is negligible.
In fact, in the case $d=1$ it reduces to the classical \textit{Lindeberg-Feller condition}
\begin{equation*}
 \lim_{m\to\infty}\max_{1\leq j\leq n_m}\sigma_{j}^2= 0\,,
\end{equation*}
 from the Lindeberg-Feller CLT (see e.g. Theorem 5.12 in \cite{Kal}). Here, we wrote $\sigma_j^2$ for $\sigma_{\{j\}}^2$.}

Our first main statement provides an explicit bound in the Wasserstein distance $d_{\rm Wass}$ for Theorem \ref{t:ordj}. We recall that, given two integrable random variables $X$ and $Y$, the Wasserstein distance between the distributions of $X$ and $Y$ is given by the quantity
$$
d_{\rm Wass}(X,Y) = \sup_{h\in {\rm Lip}(1)} \left| \E[h(X)] -\E[h(Y)] \right|, 
$$
where ${\rm Lip}(1)$ stands for the class of $1$-Lipschitz functions.

 \begin{theorem}\label{1dmt}
 As before, let $W\in L^4(\Prob)$ be a degenerate $U$-statistic of order $d$ such that \eqref{stw} is satisfied and let 
$Z\sim N(0,1)$ be a standard normal random variable. Then, it holds that 
 \begin{align*}
 d_{\rm Wass}(W,Z)&\leq \sqrt{\frac{2}{\pi}}\biggl(\E[W^4]-3+\kappa_d \rho_n^2\biggr)^{1/2}
 +\frac{2\sqrt{2}}{3} \Bigl(2\bigl(\E[W^4]-3\bigr)+3\kappa_d \rho_n^2\Bigr)^{1/2}\\
 &\leq \Bigl(\sqrt{\frac{2}{\pi}}+\frac{4}{3}\Bigr)\sqrt{\babs{\E[W^4]-3}}+\sqrt{\kappa_d}\Bigl(\sqrt{\frac{2}{\pi}}+ \frac{2\sqrt{2}}{\sqrt{3}}\Bigr)\rho_n\,.
  \end{align*}
where $\kappa_d$ is a finite constant which only depends on $d$.
\end{theorem}

Recall that a degenerate $U$-statistic $W$ of order $d$ as given by \eqref{dhom} is called \textit{symmetric}, if, additionally, the measurable spaces $(E_1,\mathcal{E}_1),\dotsc,(E_n,\mathcal{E}_n)$ all coincide, the random variables $X_1,\dotsc,X_n$ are i.i.d. and if there 
is a measurable kernel $g:E_1^d\rightarrow\R$ such that $f_J=g$ for all $J\in\D_d$. In this special situation, the relations
\begin{align*}
 1&=\Var(W)=\sum_{J\in\D_d}\E\bigl[g^2(X_j,j\in J)\bigr]=\binom{n}{d}\E\bigl[g^2(X_1,\dotsc,X_d)\bigr]\quad\text{and}\\
 \rho_n^2&=\sum_{\substack{J\in \D_d:\\ 1\in J}}\E\bigl[g^2(X_j,j\in J)\bigr]=\binom{n-1}{d-1}\E\bigl[g^2(X_1,\dotsc,X_d)\bigr]
\end{align*}
imply that 
\begin{equation*}
 \rho_n^2=\frac{d}{n}\,.
\end{equation*}
Hence, we arrive at the following corollary of Theorem \ref{1dmt}.

\begin{cor}\label{cor1d}
 Let $W\in L^4(\Prob)$ be a normalized, degenerate and symmetric $U$-statistic of order $d$ and let $Z\sim N(0,1)$ be a standard normal random variable. Then, 
  \begin{align*}
 d_{\rm Wass}(W,Z)
 &\leq \Bigl(\sqrt{\frac{2}{\pi}}+\frac{4}{3}\Bigr)\sqrt{\babs{\E[W^4]-3}}+\frac{\sqrt{d\kappa_d}}{\sqrt{n}}\Bigl(\sqrt{\frac{2}{\pi}}+ \frac{2\sqrt{2}}{\sqrt{3}}\Bigr)\,.
\end{align*}
In particular, under the assumptions of Theorem \ref{t:ordj}, a sequence $\{W_m : m\geq 1 \}$ of degenerate and symmetric $U$-statistics of a fixed order $d$ converges in distribution to $Z\sim N(0,1)$, whenever $\lim_{m\to\infty}\E[W_m^4]=3$. 
\end{cor}

\begin{remark}\label{rem1dmt}
\begin{enumerate}[(a)]
 \item The previous Therorem \ref{1dmt} is a complete quantitative counterpart to de Jong's Theorem \ref{t:ordj}. The constant $\kappa_d$ appearing in the bound is given by $\kappa_d=C_d+2d$, where $C_d$ is a combinatorial constant defined in Equation \eqref{Cd} below. 
 \item In the context of multilinear forms in independent and standardized real-valued random variables $(X_i)_{i\in\N}$ considered in \cite{NPR-aop}, the authors had to assume that the uniform moment condition 
$\sup_{i\in\N}\E[X_i^4]<\infty$ is satisfied. It is easy to check that, for homogeneous sums, this condition is in fact equivalent to the \textit{hypercontractivity condition}
\[\sup_{n\in\N}D_n<\infty\quad\text{where}\quad D_n:=\max_{J\in\D_d}\frac{\E\bigl[W_J^4\bigr]}{\sigma_J^4}\,.\]
Interestingly, this condition was also assumed in the monograph \cite{deJo89} by de Jong who was only able to dispense with it in the later paper \cite{deJo90}.
Note further that the bounds for multilinear forms in independent random variables with arbitrary distributions derived in \cite{NPR-aop} are stated in terms of three times differentiable test functions whose first 
three derivatives are uniformly bounded by a constant. Hence, our Theorem \ref{1dmt} is not only more general than the corresponding result from \cite{NPR-aop} as far as the class of random functionals dealt with is concerned but 
is also stated in terms of much less smooth test functions. 
\item It should be mentioned that the original proof of Theorem \ref{t:ordj} in \cite{deJo90} applies a quantitative martingale CLT from \cite{HeyBr70} and, by carefully revising its proof, 
one would be able to derive a bound on the rate of convergence. This issue is also briefly addressed in the introduction of the monograph \cite{deJo89} but not pursued any further.
The resulting rate, however, would be of a much worse order than the rate provided by Theorem \ref{1dmt}. Roughly, the power $1/2$ appearing in our statements would have to be systematically replaced 
by the power $1/5$. Furthermore, as was shown in \cite{Haeu88} by means of an example, the Berry-Esseen bound for martingales from \cite{HeyBr70} cannot in general be improved with respect to the rate of convergence. 
Consequently, the techniques used by de Jong are not capable of providing sharp error bounds for his qualitative statement. Note that the phenomenon of generally sharp bounds on the rate of convergence for martingale CLTs 
which reduce to sub-optimal bounds in particular situations was already discovered in the paper \cite{Bolt82}. We also stress that, unlike our work, references \cite{deJo89, deJo90} do not contain any multidimensional statements.
\end{enumerate}
\end{remark}

Finally, we would like to mention that the paper \cite{RiRo97} also deals with bounds on the normal approximation of so-called degenerate weighted $U$-statistics of order $d=2$, which have the form 
\begin{equation*}
U=\sum_{1\leq i<j\leq n}w_{i,j}\psi(X_i,X_j) 
\end{equation*}
for some vector $X=(X_1,\dotsc,X_n)$ of i.i.d. random variables, some symmetric, degenerate kernel $\psi$ and with nonnegative weights $w_{i,j}$, $1\leq i<j\leq n$.
Note that the class of weighted $U$-statistics is strictly included in our framework, since we can define the degenerate kernel $f_{\{i,j\}}$ corresponding to the subset $\{i,j\}\in\D_2$ by $f_{\{i,j\}}=w_{i,j}\psi$, leading to the Hoeffding 
components $W_{\{i,j\}}=w_{i,j}\psi(X_i,X_j)$, $1\leq i<j\leq n$. This, of course, also holds for arbitrary positive integers $d$. Note that, in contrast to our work, the bounds given in \cite{RiRo97} are expressed in terms of quantities which are related explicitly to the kernel $\psi$ and to the weights $w_{i,j}$ rather than in terms of the fourth cumulant of $U$ and, hence, cannot be immediately compared to ours.
 
\subsection{Main results, II: multivariate normal approximations}\label{intro2} In this subsection we state a new approximation theorem for the distribution of vectors of degenerate, non-symmetric $U$-statistics by a suitable multivariate normal 
distribution. 
In particular, we show that an analog of de Jong's theorem \ref{t:ordj} holds in any dimension, see Theorem \ref{t:mq}. Note that, in the multivariate case, even this qualitative result relating the asymptotic normality of the vector of degenerate, non-symmetric $U$-statistics to fourth moment conditions is completely novel.\\
As before, let $X_1,\dotsc,X_n$ be the underlying sequence of independent random variables, let 
$r\in\mathbb{N}$ and for $1\leq i\leq r$ let $W(i)$ be a random variable on $(\Omega,\F,\Prob)$ which is measurable
with respect to $\F_{[n]}=\sigma(X_1,\dotsc,X_n)$ and whose Hoeffding decomposition is given by
\begin{equation*}
W(i)=\sum_{J\in\D_{p_i}}W_J(i)
\end{equation*}
for some $p_i\in\mathbb{N}$, i.e. $W(i)$ is a degenerate $U$-statistic of order $p_i$. Without loss of generality, we can assume that
$p_i\leq p_k$ whenever $1\leq i<k\leq r$. Thus, there is an $s\in\{1,\dotsc,r\}$, positive integers $r_1,\dotsc,r_s$ with $1\leq r_1<r_2<\dotsc<r_s=r$ 
and integers $1\leq q_1<q_2<\ldots<q_s$ such that 
\begin{equation*}
 p_i=q_l\quad\text{for all}\quad i\in\{r_{l-1}+1,\dotsc,r_l\}\quad\text{and all}\quad l=1,\dotsc,s\,,
\end{equation*}
where we set $r_0:=0$.
We define 
\begin{equation*}
 W:=(W(1),\dotsc,W(r))^T
\end{equation*}
and assume that each $W(i)\in L^4(\Prob)$ with 
\[\E\bigl[W(i)\bigr]=0\quad\text{and}\quad\Var\bigl(W(i)\bigr)=\E\bigl[W(i)^2\bigr]=\sum_{J\in\D_{p_i}}\E\bigl[W_J(i)^2\bigr]=1\,,\quad 1\leq i\leq r\,.\]
We also let 
\begin{equation*}
 v_{i,k}:=\Cov\bigl(W(i),W(k)\bigr)=\E\bigl[W(i)W(k)\bigr]\,,\quad 1\leq i\leq k\leq r\,,
\end{equation*}
and 
\[\V =\V(W) :=\Cov(W)=(v_{i,k})_{1\leq i,k\leq r}\,.\]
Note that $v_{i,i}=1$ for $i=1,\dotsc,r$ and $\abs{v_{i,k}}\leq1$ for $1\leq i,k\leq r$, by the Cauchy-Schwarz inequality. Note also that $v_{i,k}=0$ unless $p_i=p_k$. Hence, $\V$ is a block diagonal matrix. 
Throughout this section we denote by 
\begin{equation*}
Z=\bigl(Z(1),\dotsc,Z(r)\bigr)^T\sim N_r(0,\V)
\end{equation*}
a centered Gaussian vector with covariance matrix $\V$.
For $1\leq k\leq r$ and $J\in\D_{p_k}$ we define 
\begin{equation*}
 \sigma_{J}(k)^2:=\Var\bigl(W_J(k)\bigr)=\E\bigl[W_J(k)^2\bigr]\quad\text{and}\quad  \rho_{n,k}^2:=\max_{1\leq j\leq n}\sum_{\substack{J\in\D_{p_k}:\\ j\in J}}\sigma_{J}(k)^2\,.
\end{equation*}
\medskip

Before stating our multivariate normal approximation theorem, we have to introduce some more notation: For a vector $x=(x_1,\dotsc,x_r)^T\in\R^r$ we denote by $\Enorm{x}$ its \textit{Euclidean norm} and for a matrix $A\in\R^{r\times r}$ we denote by $\Opnorm{A}$ 
the \textit{operator norm} induced by the Euclidean norm, i.e., 
\[\Opnorm{A}:= \sup\{ \Enorm{Ax}\,:\ \Enorm{x} = 1\}\,.\]
More generally, for a $k$-multilinear form $\psi:(\R^r)^k\rightarrow\R$, $k\in\N$, we define the \textit{operator norm}
\[\Opnorm{\psi}:=\sup\left\{\abs{\psi(u_1,\ldots,u_k)}\,:\, u_j\in\R^r,\, \Enorm{u_j}=1,\, j=1,\ldots,k\,\right\}.\]
Recall that for a function $h:\R^r\rightarrow\R$, its minimum Lipschitz constant $M_1(h)$ is given by
\[M_1(h):=\sup_{x\not=y}\frac{\abs{h(x)-h(y)}}{\Enorm{x-y}}\in[0,\infty)\cup\{\infty\}.\]
If $h$ is differentiable, then $M_1(h)=\sup_{x\in\R^r}\Opnorm{Dh(x)}$. 
More generally, for $k\geq1$ and a $(k-1)$-times differentiable function $h:\R^r\rightarrow\R$ let
\[M_k(h):=\sup_{x\not=y}\frac{\Opnorm{D^{k-1}h(x)-D^{k-1}h(y)}}{\Enorm{x-y}}\,,\]
viewing the $(k-1)$-th derivative $D^{k-1}h$ of $h$ at any point $x$ as a $(k-1)$-multilinear form.
Then, if $h$ is actually $k$-times differentiable, we have $M_k(h)=\sup_{x\in\R^r}\Opnorm{D^kh(x)}$. 
Having in mind this identity, we define  $M_0(h):=\fnorm{h}$.\\
Recall that, for two matrices $A,B\in\R^{r\times r}$, their \textit{Hilbert-Schmidt inner product} is defined by 
\begin{eqnarray*}
\langle A,B\rangle_{\HS}:=\Tr\bigl(AB^T\bigr)=\Tr\bigl(BA^T\bigr)=\Tr\bigl(B^TA\bigr)=\sum_{i,j=1}^r a_{ij}b_{ij}\,.
\end{eqnarray*}
Thus, $\langle\cdot,\cdot\rangle_{\HS}$ is just the standard inner product on $\R^{r\times r}\cong \R^{r^2}$.
The corresponding \textit{Hilbert-Schmidt norm} will be denoted by $\HSnorm{\cdot}$.  
With this notion at hand, following \cite{ChaMe08} and \cite{Meck09}, for $k=2$ we finally define 
\[\tilde{M}_2(h):=\sup_{x\in\R^r}\HSnorm{\Hess h(x)}\,,\]
where $\Hess h$ is the \textit{Hessian matrix} corresponding to $h$.

\begin{theorem}\label{mdmt}
There exist finite constants $C_{q_l}$, $1\leq l\leq s$, only depending on $q_l$ as well as finite constants $C_{i,k}$, $1\leq i,k\leq r$, depending on $i$ and $k$ only through $p_i$ and $p_k$ such that, with the definition
\begin{align*}
 A&:=4\sum_{l=1}^s\frac{q_l^2}{q_1^2}\sum_{i,k=r_{l-1}+1}^{r_l}\Bigl(\E\bigl[W(i)^2W(k)^2\bigr]-\E\bigl[Z(i)^2Z(k)^2\bigr]\notag\\
&\hspace{2cm}+q_l\min\Bigl(\rho_{n,k}^2\,,\,\rho_{n,i}^2\Bigr)+q_l\rho_{n,k}\rho_{n,i}+C_{i,k}\max\bigl(\rho_{n,i}^2,\rho_{n,k}^2\bigr)\Bigr)\notag\\
&\;+2\sum_{1\leq l<m\leq s}\frac{(q_l+q_m)^2}{q_1^2}\sum_{i=r_{l-1}+1}^{r_l}\sum_{k=r_{m-1}+1}^{r_m}\Biggl[\Bigl(\E\bigl[W(i)^4\bigr]-1\Bigr)^{1/2}\notag\\
&\hspace{2cm}\Bigl(\E\bigl[W(k)^4\bigr]-3+\bigl(2q_m+C_{q_m}\bigr)\rho_{n,k}^2\Bigr)^{1/2}\notag\\
 &\hspace{3cm}+\min\Bigl(q_l \rho_{n,k}^2\,,\,q_m \rho_{n,i}^2\Bigr)+C_{i,k}\max\bigl(\rho_{n,i}^2,\rho_{n,k}^2\bigr)\Biggr]
\end{align*}
and under the above assumptions, we have the following bounds:
 \begin{enumerate}[{\normalfont(i)}]
  \item For any $h\in C^3(\R^r)$ such that $\E\bigl[\abs{h(W)}\bigr]<\infty$ and $\E\bigl[\abs{h(Z)}\bigr]<\infty$, 
\begin{align*}
&\babs{\E[h(W)]-\E[h(Z)]}\leq \frac{1}{4q_1}\tilde{M}_2(h)\sqrt{A}\\
&\;+\frac{\sqrt{2r}}{9}M_3(h)\sum_{l=1}^s\frac{q_l}{q_1}\sum_{i=r_{l-1}+1}^{r_l}\Bigl(2\bigl(\E\bigl[W(i)^4\bigr]-3\bigr)+ 3\bigl(C_{q_l}+2q_l\bigr)\rho_{n,i}^2\Bigr)^{1/2}\,.
\end{align*}
\item If $\V$ is in addition positive definite, then for each $h\in C^2(\R^r)$ such that\\ $\E\bigl[\abs{h(W)}\bigr]<\infty$ and $\E\bigl[\abs{h(Z)}\bigr]<\infty$, 
\begin{align*}
&\babs{\E[h(W)]-\E[h(Z)]}\leq  \frac{1}{\sqrt{2\pi}q_1}M_1(h)\Opnorm{\V^{-1/2}}\sqrt{A}\\
&\;+\frac{\sqrt{\pi r}}{6}M_2(h)\Opnorm{\V^{-1/2}}\sum_{l=1}^s\frac{q_l}{q_1}\sum_{i=r_{l-1}+1}^{r_l}\Bigl(2\bigl(\E\bigl[W(i)^4\bigr]-3\bigr)+ 3\bigl(C_{q_l}+2q_l\bigr)\rho_{n,i}^2\Bigr)^{1/2}\,.
\end{align*}
\end{enumerate}
\end{theorem}

Fix $r\in\N$. Since the class of all compactly supported, three times differentiable functions $h$ on $\R^r$ is convergence-determining, from Theorem \ref{mdmt} (i) we obtain the following statement, which is a new multidimensional extension of Theorem \ref{t:ordj}.

\begin{theorem}\label{t:mq} Fix $r\geq 2$, as well as integers $p_1,...,p_r$, and let $n_m\to \infty$, as $m\to \infty$. Let $W_m:=(W_m(1),\dotsc,W_m(r))^T$, $m\geq 1$, be a sequence of random vectors such that each $W_m(k)$ is a
centered, unit variance degenerate $U$-statistic of order $p_k$, whose argument is the vector of independent random elements $(X_1^{(m)},...,X_{n_m}^{(m)})$. 
Furthermore, let $\mathbf{\Sigma}\in\R^{r\times r}$ be a positive semi-definite matrix with $\mathbf{\Sigma}(j,j)=1$ for $j=1,\dotsc,r$ and denote by $N= (N(1),...,N(r))^T\sim N_r(0,\mathbf{\Sigma})$ a centered Gaussian vector with covariance matrix $\mathbf{\Sigma}$.
Assume the following:
\begin{itemize}
\item[\rm (i)]  The covariance matrix of $W_m$ converges to $\mathbf{\Sigma}$;


\item[\rm (ii)] As $m\to \infty$, $\rho_{n_m,k}^2\to 0$, for every $k=1,...,r$;

\item[\rm (iii)] As $m\to \infty$, $\E[W_m(k)^4] \to 3$, for every $k=1,...,r$;

\item[\rm (iv)] If $j\not=k$ but $p_j = p_k$ then, as $m\to \infty$,\\
$\E[ W_m(j)^2W_m(k)^2] \to \E[N(j)^2N(k)^2]=1+(\mathbf{\Sigma}(j,k))^2$.
\end{itemize}
Then, as $m\to \infty$, $W_m$ converges in distribution to $N$.

\end{theorem}

In the framework of the normal approximation of vectors of eigenfunctions of diffusive Markov semigroups, a condition similar to (iv) in the above statement has been recently introduced and applied in \cite{CNPP}. The rest of the paper is organized as follows: Section 2 contains the proof of our one-dimensional result, Section 3 focusses on our multidimensional statements, whereas Section 4 contains the detailed proofs of several technical lemmas.


\section{Proof of the one-dimensional theorem}\label{1dim}
In this section we give a detailed proof of Theorem \ref{1dmt}. First we review Stein's method of exchangeable pairs for univariate normal approximation. 
\subsection{Stein's method of exchangeable pairs}\label{1dimex}
The exchangeable pairs approach within Stein's method dates back to Stein's celebrated monograph \cite{St86}. Recall that a pair $(X,X')$ of random elements on a common probability space is called \textit{exchangeable}, if 
\begin{equation*}
 (X,X')\stackrel{\D}{=}(X',X)\,.
 \end{equation*}
In \cite{St86} C. Stein extensively illustrated the fact that a given normalized random variable $W$ is close in distribution to $Z\sim N(0,1)$, whenever one can construct another random variable $W'$ on the same space such that: (i) $W'$ is `close' to $W$ in some proper, quantifiable sense, (ii) the pair $(W,W')$ is exchangeable, (iii) the \textit{linear regression property}
\begin{equation}\label{linregprop}
 \E\bigl[W'-W\,\bigl|\,W\bigr]=-\lambda W
\end{equation}
is satisfied for some small $\lambda>0$, and (iv) the conditional second moment of $W'-W$ given $W$ is close to its mean, the constant $2\lambda$, in the $L^1$ metric.
For a precise statement see Theorem \ref{1dimplugin} below.\\
The range of examples to which this method can be applied was considerably extended by the work \cite{RiRo97} by Rinott and Rotar, who proved bounds on the distance to normality under the 
condition that the linear regression property is only approximately satisfied, i.e. that there is some negligible remainder term $R$ such that 
\begin{equation}\label{genlinreg}
\frac{1}{\lambda}\E\bigl[W'-W\,\bigl|\,\mathcal{G}\bigr]=-W+R
\end{equation}
is satisfied, where $\mathcal{G}$ is a sub-$\sigma$-field of $\F$ such that $\sigma(W)\subseteq\mathcal{G}$. The method of exchangeable pairs has been generalized to other absolutely continuous distributions, like the exponential (\!\!\cite{CFR11} and \cite{FulRos13}), the multivariate normal 
(\!\!\cite{ChaMe08}, \cite{ReiRoe09} and \cite{Meck09}) and the Beta distribution \cite{DoeBeta}. It has also been developed for general classes of one-dimensional absolutely continuous distributions 
in \cite{ChSh}, \cite{EiLo10} and \cite{DoeBeta}. As was observed in \cite{Ro08}, in the case of one-dimensional distributional approximation one may in general relax the exchangeability condition to the assumption that 
$W$ and $W'$ be identically distributed.\\
In this article we focus on the exchangeable pairs method in the context of one- and multidimensional normal approximation. The following result is a variant of Theorem 1, Lecture 3 in \cite{St86} (see also Theorem 4.9 in \cite{CGS}). It slightly improves on these result with respect to the constants appearing in the bound and is also stated in terms of identically distributed random variables $W,W'$ as opposed to exchangeable ones as well as for general sub-$\sigma$-fields $\mathcal{G}$ 
of $\F$ with $\sigma(W)\subseteq\mathcal{G}$. The proof is standard and therefore omitted from the paper. Moreover, the result is a direct consequence of Proposition 3.19 in \cite{DoeBeta} together with the best known bounds 
on the first two derivatives of the solution to the standard normal Stein equation for Lipschitz test functions (see e.g. Lemma 2.4 in \cite{CGS}).  

\begin{theorem}\label{1dimplugin}
 Let $(W,W')$ be a pair of identically distributed, square-integrable random variables on $(\Omega,\F,\Prob)$ such that, for some $\lambda>0$, \eqref{linregprop} holds. Furthermore, let $\mathcal{G}$ be a sub-$\sigma$-field of $\F$ with $\sigma(W)\subseteq\mathcal{G}$. Then, we have the bound 
\begin{align}\label{1dexbound}
 d_{\rm Wass}(W,Z)&\leq\sqrt{\frac{2}{\pi}}\sqrt{\Var\Bigl(\frac{1}{2\lambda}\E\bigl[(W'-W)^2\,\bigl|\,\mathcal{G}\bigr]\Bigr)} +\frac{1}{3\lambda}\E\babs{W'-W}^3\,.
\end{align}
\end{theorem}

For the proof of Theorem \ref{1dmt} we will need the following new auxiliary result about exchangeable pairs satisfying identity \eqref{linregprop} which might be of independent interest.

\begin{lemma}\label{exlemma}
Let $(W,W')$ be an exchangeable pair of real-valued random variables in $L^4(\Prob)$ such that, for some $\lambda>0$, \eqref{linregprop} is satisfied and let $\mathcal{G}$ be a sub-$\sigma$-field of $\F$ with $\sigma(W)\subseteq\mathcal{G}$. Then, 
 \begin{align*}
  \frac{1}{4\lambda}\E\bigl[(W'-W)^4\bigr]&=3\E\Bigl[W^2\frac{1}{2\lambda}\E\bigl[(W'-W)^2\,\bigl|\,\G\bigr]\Bigr]
	-\E\bigl[W^4\bigr]\,.
	\end{align*}
\end{lemma}

\begin{proof}
By exchangeability of $(W,W')$ we have 
\begin{align}\label{el1}
\frac12\E\bigl[(W'-W)^4\bigr]&=\E\bigl[W( W-W')^3\bigr]=\E\Bigl[W^4-3W^3W'+3(W'W)^2-W(W')^3\Bigr]\notag\\
&=\E[W^4]+3\E\bigl[(WW')^2\bigr]-4\E\bigl[W^3W'\bigr]\,.
\end{align}
Also, by \eqref{genlinreg}
\begin{align}\label{el2}
\E\bigl[W^3W'\bigr]&=\E\Bigl[W^3\E\bigl[W'\,\bigl|\,\G\bigl]\Bigr]=(1-\lambda)\E[W^4]
\end{align}
and
\begin{align}\label{el3}
\E\bigl[(WW')^2\bigr]&=\Bigl[W^2\E\bigl[(W'-W+W)^2\,\bigl|\,\G\bigr]\Bigr]\notag\\
&=\E\Bigl[W^2\E\bigl[(W'-W)^2+2W(W'-W)+W^2\,\bigl|\,\G\bigr]\Bigr]\notag\\
&=\E[W^4]-2\lambda\E[W^4]+\E\Bigl[W^2\E[(W'-W)^2\,\bigl|\,\G\bigr]\Bigr]\notag\\
&=(1-2\lambda)\E[W^4]+\E\Bigl[W^2\E[(W'-W)^2\,\bigl|\,\G\bigr]\Bigr]\,.
\end{align}
Thus, from \eqref{el1}, \eqref{el2} and \eqref{el3} we obtain that 
\begin{align*}
\frac12\E\bigl[(W'-W)^4\bigr]&=\Bigl(1+3(1-2\lambda)-4(1-\lambda)\Bigr)\E[W^4]+3\E\Bigl[W^2\E[(W'-W)^2\,\bigl|\,\G\bigr]\Bigr]\\
&=3\E\Bigl[W^2\E[(W'-W)^2\,\bigl|\,\G\bigr]-2\lambda\E[W^4]\,,
\end{align*}
proving the lemma.\\
\end{proof}

\subsection{Proof of Theorem \ref{1dmt}}\label{proof1d}
Let $W\in L^4(\Prob)$ be as in Theorem \ref{1dmt} such that its Hoeffding decomposition is given by \eqref{dhom}.
We are going to apply Theorem \ref{1dimplugin} to the $\sigma$-field $\mathcal{G}=\sigma(X_1,\dots,X_n)$ and to the exchangeable pair $(W,W')$ which is constructed as follows:  
 Let $Y:=(Y_j)_{1\leq j\leq n}$ be an independent copy of $X:=(X_j)_{1\leq j\leq n}$ and let $\alpha$ be uniformly distributed 
on $\{1,\dotsc,n\}$ such that $X,Y$ and $\alpha$ are jointly independent. Letting, for $j=1,\dotsc,n$, 
\begin{equation*}
 X_j':=\begin{cases}
        Y_j\,,&\text{if } \alpha=j\\
        X_j\,,&\text{if }\alpha\not=j
       \end{cases}
\end{equation*}
and
\begin{equation*}
 X':=(X_1',\dotsc,X_n')
\end{equation*}
it is easy to see that the pair $(X,X')$ is exchangeable. Finally, as exchangeability is preserved under functions, defining
\begin{align*}
 W':=f(X_1',\dotsc,X_n')&=\sum_{j=1}^n 1_{\{\alpha=j\}}\Biggl(\sum_{\substack{J\in\D_d:\\j\notin J}}W_J
+\sum_{\substack{J\in\D_d:\\j\in J}}W_J^{(j)}\Biggr)\\
 &=:\sum_{\substack{J\in\D_d:\\ \alpha\notin J}}W_J+\sum_{\substack{J\in\D_d:\\ \alpha\in J}}W_J^{(\alpha)}\,,
\end{align*}
also the pair $(W,W')$ is exchangeable. Here, for $J=\{j_1,\dotsc,j_d\}\in\D_d$ with $1\leq j_1<j_2<\dotsc<j_d\leq n$ and $j=j_k\in J$, we write
\begin{equation*}
 W_J^{(j)}:=f_J(X_{j_1},\dotsc,X_{j_{k-1}},Y_{j_k},X_{j_{k+1}},\dotsc,X_{j_d})\,,
\end{equation*}
where the kernel $f_J$ is given by \eqref{kernelsfj}.
We now show that the pair $(W,W')$ satisfies Stein's linear regression property \eqref{linregprop} exactly with coefficient $\lambda=d/n$.

\begin{lemma}\label{linreg}
With the above definitions, we have 
\begin{equation*}
 \E\bigl[W'-W\,\bigl|\,W\bigr]=\E\bigl[W'-W\,\bigl|\,X\bigr]=-\frac{d}{n}W\,.
\end{equation*}
\end{lemma}

\begin{proof}
 It suffices to prove the second equality. Note that 
 \begin{equation*}
  W'-W=\sum_{j=1}^n 1_{\{\alpha=j\}}\sum_{\substack{J\in\D_d:\\j\in J}}\Bigl(W_J^{(j)}-W_J\Bigr)=
	\sum_{\substack{J\in\D_d:\\\alpha\in J}}\Bigl(W_J^{(\alpha)}-W_J\Bigr)\,.
 \end{equation*}
Hence, by independence, 
\begin{align*}
 \E\bigl[W'-W\,\bigl|\,X\bigr]&=\frac{1}{n}\sum_{j=1}^n\sum_{J:j\in J}\Bigl(\E\bigl[W_J^{(j)}\,\bigl|\,X\bigr]-W_J\Bigr)\\
 &=\frac{1}{n}\sum_{j=1}^n\sum_{J:j\in J}\Bigl(\E\bigl[W_J^{(j)}\,\bigl|\,X_i,\,i\in J\setminus\{j\}\bigr]-W_J\Bigr)\\
 &=\frac{1}{n}\sum_{j=1}^n\sum_{J:j\in J}\Bigl(\E\bigl[W_J\,\bigl|\,\F_{J\setminus\{j\}}\bigr]-W_J\Bigr)\\
 &=-\frac{1}{n}\sum_{j=1}^n\sum_{J:j\in J}W_J=-\frac{1}{n}\sum_{J\in \D_d}W_J\sum_{j\in J}1\\
 &=-\frac{d}{n}\sum_{J\in \D_d}W_J=-\frac{d}{n}W\,.
\end{align*}
Here, we have used the defining property of the Hoeffding decomposition to obtain the fourth equality.\\
\end{proof}

We would like to mention that the same construction of the exchangeable pair $(W,W')$ was used in \cite{RiRo97} in the 
situation of weighted $U$-statistics. They also noted the validity of \eqref{linregprop} with 
$\lambda=d/n$ in the special case of completely degenerate weighted $U$-statistics of order $d$.\\
In order to apply \eqref{1dexbound}, by Lemma \ref{linreg}, we thus have to compute an upper bound on the variance of $\frac{n}{2d}\E\bigl[(W'-W)^2\,\bigl|\,X\bigr]$. This is done by finding the Hoeffding 
decomposition of this quantity in terms of the Hoeffding decomposition of $W^2$ for which we will now find a new convenient expression.
More generally, we derive a formula for the Hoeffding decomposition of the product of two degenerate $U$-statistics, which will also be needed for the proof of Theorem \ref{mdmt}.

Assume that $1\leq p,q\leq n$ and that $W$ and $V$ are square-integrable $p$- and $q$-degenerate $U$-statistics with respect to the same underlying sequence $X$, respectively, with Hoeffding decompositions 
\begin{equation}\label{hdvw}
 W=\sum_{J\in\D_p}W_J\quad\text{and}\quad V=\sum_{K\in\D_q}V_K\,.
\end{equation}
The product $U:=VW$ in general is not a degenerate $U$-statistic, but it clearly has a Hoeffding decomposition of the form
\begin{equation}\label{hdu}
 U=\sum_{M\subseteq [n]}U_M=\sum_{\substack{M\subseteq [n]:\\\abs{M}\leq p+q}} U_M\,.
\end{equation}

The following simple observation will be crucial for the computation of the Hoeffding decompositions of both $VW$ and of the quantity $\frac{n}{2d}\E\bigl[(W'-W)^2\,\bigl|\,X\bigr]$.
\begin{lemma}\label{hdle1}
If $L\subseteq[n]$ is such that $J\Delta K=(J\setminus K)\cup (K\setminus J)\not\subseteq L$, then 
\[\E\bigl[W_JV_K\,\bigl|\,\F_L\bigr]=0\,.\]
\end{lemma}

\begin{proof}
Assume e.g. that $(J\setminus K)\setminus L= J\setminus(K\cup L)\not=\emptyset$. Then, 
\begin{align*}
\E\bigl[W_JV_K\,\bigl|\,\F_L\bigr]&=\E\Bigl[V_K\E\bigl[W_J\,\bigl|\,\F_{K\cup L}\bigr]\,\Bigl|\,\F_L\Bigr]
=\E\bigl[V_K\cdot0\,\bigl|\,\F_L\bigr]=0\,,
\end{align*}
as $\E\bigl[W_J\,\bigl|\,\F_{K\cup L}\bigr]=0$ because $J\not\subseteq K\cup L$.
\end{proof}

\begin{lemma}\label{hdle2}
Let $J\in\D_p$ and $K\in\D_q$, respectively. 
\begin{enumerate}[{\normalfont (a)}]
\item The Hoeffding decomposition of $W_JV_K$ is given by 
\begin{equation}\label{hdwjvk}
W_JV_K=\sum_{\substack{M\subseteq[n]:\\ J\Delta K\subseteq M\subseteq J\cup K}}\sum_{\substack{L\subseteq[n]:\\ J\Delta K\subseteq L\subseteq M}}(-1)^{\abs{M}-\abs{L}}\E\bigl[W_JV_K\,\bigl|\,\F_L\bigr]\,.
\end{equation}
\item If $j\in J\cap K$, then we have the Hoeffding decomposition
\begin{equation}\label{hdwjvkcond}
\E\bigl[W_JV_K\,\bigl|\,\F_{(J\cup K)\setminus\{j\}}\bigr]=\sum_{\substack{M\subseteq[n]:\\ J\Delta K\subseteq M\subseteq (J\cup K)\setminus\{j\}}}\sum_{\substack{L\subseteq[n]:\\ J\Delta K\subseteq L\subseteq M}}(-1)^{\abs{M}-\abs{L}}\E\bigl[W_JV_K\,\bigl|\,\F_L\bigr]\,.
\end{equation}
\end{enumerate}
\end{lemma}

\begin{proof}
The claim of (a) follows immediately from Lemma \ref{hdle1} and from the general formula for the Hoeffding decomposition of an $\F_{J\cup K}$-measurable random variable $T$ which is given by 
\[T=\sum_{M\subseteq J\cup K}\Bigl(\sum_{L\subseteq M}(-1)^{\abs{M}-\abs{L}}\E[T|\F_L] \Bigr)\,.\]
The claim of (b) follows similarly upon observing that, for $ L\subseteq (J\cup K)\setminus\{j\}$ we have 
\[\E\Bigl[\E\bigl[W_JV_K\,\bigl|\,\F_{(J\cup K)\setminus\{j\}}\bigr]\,\Bigl|\,\F_L\Bigr]
=\E\bigl[W_JV_K\,\bigl|\,\F_L\bigr]\,.\]
\end{proof}

The next result which might be of independent interest plays a similar role as the product formula for two multiple Wiener-It\^{o} integrals (see e.g. \cite{NouPecbook}).
\begin{theorem}[Product formula for degenerate $U$-statistics]\label{prodform}
Let $1\leq p,q\leq n$ and let $W,V\in L^2(\Prob)$ be $p$- and $q$-degenerate $U$-statistics, respectively, with respective Hoeffding decompositions given by \eqref{hdvw}. 
Then, the Hoeffding decomposition \eqref{hdu} of $U:=VW$ is given by the following formula:
\begin{align*}
VW&=\sum_{\substack{M\subseteq[n]:\\\abs{M}\leq p+q}}\biggl(\sum_{\substack{J\in\D_p,K\in\D_q:\\J\Delta K\subseteq M\subseteq J\cup K}}
\sum_{\substack{L\subseteq[n]:\\ J\Delta K\subseteq L\subseteq M}}(-1)^{\abs{M}-\abs{L}}\E\bigl[W_JV_K\,\bigl|\,\F_L\bigr]\biggr)\\
&=\sum_{\substack{M\subseteq[n]:\\\abs{M}\leq p+q}}\biggl(\sum_{L\subseteq M}(-1)^{\abs{M}-\abs{L}}
\sum_{\substack{J\in\D_p,K\in\D_q:\\J\Delta K\subseteq L,\\
M\subseteq J\cup K}}\E\bigl[W_JV_K\,\bigl|\,\F_L\bigr]\biggr)\,,
\end{align*}
i.e. for $M\subseteq[n]$ with $\abs{M}\leq p+q$ we have 
\begin{align*}
U_M&=\sum_{\substack{J\in\D_p,K\in\D_q:\\J\Delta K\subseteq M\subseteq J\cup K}}
\sum_{\substack{L\subseteq[n]:\\ J\Delta K\subseteq L\subseteq M}}(-1)^{\abs{M}-\abs{L}}\E\bigl[W_JV_K\,\bigl|\,\F_L\bigr]\\
&=\sum_{L\subseteq M}(-1)^{\abs{M}-\abs{L}}
\sum_{\substack{J\in\D_p,K\in\D_q:\\J\Delta K\subseteq L,\\
M\subseteq J\cup K}}\E\bigl[W_JV_K\,\bigl|\,\F_L\bigr]\,.
\end{align*}
\end{theorem}

\begin{proof}
By the linearity of the Hoeffding decomposition and since we have 
\begin{equation*}
 VW=\sum_{J\in\D_p,K\in\D_q}W_JV_K\,,
\end{equation*}
it suffices to collect the terms resulting from the Hoeffding decompositions of the summands $W_JV_K$ in a suitable way. By Lemma \ref{hdle2} (a) we have
\begin{align*}
VW&=\sum_{J\in\D_p,K\in\D_q}W_JV_K=\sum_{J\in\D_p,K\in\D_q}\sum_{\substack{M\subseteq[n]:\\ J\Delta K\subseteq M\subseteq J\cup K}}\sum_{\substack{L\subseteq[n]:\\ J\Delta K\subseteq L\subseteq M}}(-1)^{\abs{M}-\abs{L}}\E\bigl[W_JV_K\,\bigl|\,\F_L\bigr]\\
&=\sum_{\substack{M\subseteq[n]:\\\abs{M}\leq p+q}}\biggl(\sum_{L\subseteq M}(-1)^{\abs{M}-\abs{L}}
\sum_{\substack{J\in\D_p,K\in\D_q:\\J\Delta K\subseteq L,\\
M\subseteq J\cup K}}\E\bigl[W_JV_K\,\bigl|\,\F_L\bigr]\biggr)\,.
\end{align*}
\end{proof}

Now we are in the position to express the Hoeffding decomposition of\\ $\frac{n}{2d}\E\bigl[(W'-W)^2\,\bigl|\,X\bigr]$ in terms of that of $W^2$.
Since we prove a more general result, Lemma \ref{condexmult} below, we do not give its proof, here. 
\begin{lemma}\label{condex}
Let $W^2=\sum_{\abs{M}\leq 2d}U_M$ be the Hoeffding decomposition of $W^2$. Then, we have the Hoeffding decomposition
\begin{align*}
\frac{n}{2d}\E\bigl[(W'-W)^2\,\bigl|\,X\bigr]&=\sum_{\substack{M\subseteq[n]:\\\abs{M}\leq 2d-1}}a_MU_M\,,
\end{align*}
with 
\[a_M=1-\frac{\abs{M}}{2d}\in[0,1]\quad\text{for each}\quad M\subseteq[n]\text{ with }\abs{M}\leq2d\,.\]
\end{lemma}

Before we proceed, let us, following \cite{deJo89} and \cite{deJo90}, introduce the following important classes of quadruples $(J_1,J_2,J_3,J_4)\in\D_d^4$. We call an element $j\in J_1\cup J_2\cup J_3\cup J_4$ a 
\textit{free index}, if it appears in $J_i$ for exactly one $i\in\{1,2,3,4\}$. 
Note that this implies that
\begin{equation}\label{freeind}
 \E\bigl[W_{J_1}W_{J_2}W_{J_3}W_{J_4}\bigr]=0
\end{equation}
by Lemma \ref{hdle} 4).
We say that $(J_1,J_2,J_3,J_4)$ is \textit{bifold}, if each element 
in the union $J_1\cup J_2\cup J_3\cup J_4$ appears in $J_i$ for exactly two values of $i\in\{1,2,3,4\}$, i.e. if
\begin{equation*}
 1_{J_1}+1_{J_2}+1_{J_3}+1_{J_4}=2 \cdot1_{J_1\cup J_2\cup J_3\cup J_4}\,.
\end{equation*}
Let us denote by $\B=\B_d$ the set of all bifold quadruples. 
Among the bifold quadruples, the most important ones are given by the subclass $\mathcal{S}_0$ which is defined by 
\begin{align*}
 \mathcal{S}_0&=\Bigl\{(J,K,L,M)\in\D_d^4\,:\, J\cap K= L\cap M=\emptyset\,,\quad \emptyset\subsetneq J\cap L=J\setminus(J\cap M)\subsetneq J\\
 &\hspace{2cm}\text{and } \emptyset\subsetneq K\cap L=K\setminus(K\cap M)\subsetneq K\Bigr\}\,.
\end{align*}
Further, we denote by $\mathcal{T}=\mathcal{T}_d$ the set of all quadruples $(J_1,J_2,J_3,J_4)\in\D_d^4$ that are neither bifold nor have a free index. This just means that 
\begin{equation*}
 1_{J_1}+1_{J_2}+1_{J_3}+1_{J_4}\geq2 \cdot1_{J_1\cup J_2\cup J_3\cup J_4}
\end{equation*}
and there exists at least one $j\in[n]$ such that
\begin{equation*}
 1_{J_1}(j)+1_{J_2}(j)+1_{J_3}(j)+1_{J_4}(j)\geq3\,,
\end{equation*}
i.e. each element of the union $J_1\cup J_2\cup J_3\cup J_4$ appears in $J_i$ for at least two values of $i\in\{1,2,3,4\}$ and there is an element of the union $J_1\cup J_2\cup J_3\cup J_4$ that appears in $J_i$ for at least three values of $i\in\{1,2,3,4\}$.\\
Following \cite{deJo90} let us define the quantities 
\begin{align*}
S_0&:=\sum_{\substack{J,K,L,M\in\D_d:\\ J\cap K=\emptyset=L\cap M,\\\emptyset\subsetneq J\cap L\subsetneq J\,,\\\emptyset\subsetneq J\cap M\subsetneq J}}\E\bigl[W_JW_KW_LW_M\bigr]
=\sum_{(J,K,L,M)\in\mathcal{S}_0}\E\bigl[W_JW_KW_LW_M\bigr]\,,
\end{align*}
as well as
\begin{equation*}
 \tau:=\tau_d:=\sum_{(J,K,L,M)\in\mathcal{T}}\sigma_J\sigma_K\sigma_L\sigma_M\,.
\end{equation*}
Note that the last identity in the definition of $S_0$ is true by virtue of \eqref{freeind}.
The following result is Proposition 5 (b) of \cite{deJo90}. We will prove a more general version stated as Proposition \ref{s0propgen} to deal with the multivariate case. 

\begin{prop}\label{s0prop}
 We have 
 \begin{equation*}
S_0\geq -\tau\,.
\end{equation*}
\end{prop}

Recall the definition of the Lindeberg-Feller quantity $\rho=\rho_n$ given in \eqref{rhodef}. Next, we state a substantial improvement of Lemma B in \cite{deJo90}. Indeed, there the upper bound on $\tau$ is of order $\rho$ as compared to the order $\rho^2$ which we obtain.
Its proof is deferred to Section \ref{proofs}.
 \begin{prop}\label{taubound}
  For each $d\in\mathbb{N}$ there is a finite constant $C_d$ which is independent of $n$ such that 
  \begin{equation*}
   \tau=\sum_{(J,K,L,M)\in\mathcal{T}}\sigma_J\sigma_K\sigma_L\sigma_M\leq C_d \rho^2\,.
  \end{equation*}
  Furthermore, we can let $C_2=13$.
 \end{prop}
 
The next two lemmas will be very useful for what follows.

\begin{lemma}\label{varlemma}
Again, let $W^2=\sum_{\abs{M}\leq 2d}U_M$ denote the Hoeffding decomposition of $W^2$. Then, we have the bound 
\begin{equation*}
 \sum_{\substack{M\subseteq[n]:\\\abs{M}\leq 2d-1}}\Var(U_M) \leq\E\bigl[W^4\bigr]-3+\kappa_d \rho^2\,,
\end{equation*}
where $\kappa_d=C_d+2d$ and $C_d$ is the constant from Proposition \ref{taubound}.
\end{lemma}

\begin{proof}
 We have 
\begin{align}\label{t12}
 \sum_{\substack{M\subseteq[n]:\\\abs{M}\leq 2d-1}}\Var(U_M)&=\Var(W^2)-\sum_{\substack{M\subseteq[n]:\\\abs{M}= 2d}}\Var(U_M)\notag\\
 &=\E\bigl[W^4\bigr]-1-\sum_{\substack{M\subseteq[n]:\\\abs{M}= 2d}}\Var(U_M)\notag\\
 &=\E\bigl[W^4\bigr]-3+2-\sum_{\substack{M\subseteq[n]:\\\abs{M}= 2d}}\E[U_M^2]\notag\\
 &=\E\bigl[W^4\bigr]-3+\Bigl( 2-\sum_{\substack{J,K,L,M\in\D_d:\\ J\cap K=\emptyset=L\cap M}}\E\bigl[W_JW_KW_LW_M\bigr]\Bigr)\,.
\end{align}

For the last equality we have used the fact that for $\abs{M}=2d$ we have 
\begin{equation*}
 U_M=\sum_{\substack{J,K\in\D_d:\\J\cap K=\emptyset,\\ J\cup K=M}}W_JW_K\,.
\end{equation*}
Also, we can write 
\begin{align*}
\sum_{\substack{J,K,L,M\in\D_d:\\ J\cap K=\emptyset=L\cap M}}\E\bigl[W_JW_KW_LW_M\bigr]
&=2\sum_{\substack{J,K\in\D_d:\\ J\cap K=\emptyset}}\E\bigl[W_J^2W_K^2\bigr]+S_0\\
&=2\sum_{\substack{J,K\in\D_d:\\ J\cap K=\emptyset}}\sigma_J^2\sigma_K^2+S_0\\
&=2-2\sum_{\substack{J,K\in\D_d:\\ J\cap K\not=\emptyset}}\sigma_J^2\sigma_K^2+S_0\\
&\geq 2-2d \rho^2+S_0\,,
\end{align*}
where we have used Lemma \ref{sigmale} to obtain the last inequality.
Thus, from \eqref{t12} and Propositions \ref{s0prop} and \ref{taubound} we conlude that 
\begin{align}\label{compt1}
 \sum_{\substack{M\subseteq[n]:\\\abs{M}\leq 2d-1}}\Var(U_M)
 &= \Var(W^2)-\sum_{\substack{M\subseteq[n]:\\\abs{M}= 2d}}\Var(U_M)\\
 &\leq\E\bigl[W^4\bigr]-3+2d\rho^2-S_0\notag\\
 &\leq\E\bigl[W^4\bigr]-3+(2d+C_d)\rho^2\notag\\
 &=\E\bigl[W^4\bigr]-3+\kappa_d \rho^2\notag\,,
\end{align}
which proves the claim.\\
\end{proof}

 Now we are able to bound the first term on the right hand side of \eqref{1dexbound}:

\begin{lemma}\label{1dterm1}
  For the above constructed exchangeable pair we have 
 \begin{equation}\label{bound1}
\Var\Bigl(\frac{n}{2d}\E\bigl[(W'-W)^2\,\bigl|\,X\bigr]\Bigr)\leq \E\bigl[W^4\bigr]-3+\kappa_d\rho^2\,.
\end{equation}
\end{lemma}

\begin{proof}
Using the orthogonality of the summands within the Hoeffding decomposition as well as $a_M\in[0,1]$, $\abs{M}\leq 2d-1$, from Lemma \ref{condex} we obtain that
\begin{align*}
 \Var\Bigl(\frac{n}{2d}\E\bigl[(W'-W)^2\,\bigl|\,X\bigr]\Bigr)&=\Var\biggl(\sum_{\substack{M\subseteq[n]:\\\abs{M}\leq 2d-1}}a_M U_M\biggr)
 =\sum_{\substack{M\subseteq[n]:\\\abs{M}\leq 2d-1}}a_M^2\Var(U_M)\\
 &\leq \sum_{\substack{M\subseteq[n]:\\\abs{M}\leq 2d-1}}\Var(U_M)\\
 &\leq \E\bigl[W^4\bigr]-3+\kappa_d \rho^2\,,
\end{align*}
where the final inequality is by Lemma \ref{varlemma}.\\
\end{proof}

Now, we proceed to bounding the second error term appearing in the bound \eqref{1dexbound} from Theorem \ref{1dimplugin}. The next lemma will be crucial for doing this.

\begin{lemma}\label{remlemma}
 For the above constructed exchangeable pair we have the bound
 \[\frac{n}{4d}\E\bigl[(W'-W)^4\bigr]\leq 2\bigl(\E[W^4]-3\bigr)+3\kappa_d \rho^2\,.\]
\end{lemma}

\begin{proof}
From Lemmas \ref{linreg}, \ref{exlemma} and \ref{condex} we obtain that 
\begin{align*}
 \frac{n}{4d}\E\bigl[(W'-W)^4\bigr]&=3\E\Bigl[W^2\frac{n}{2d}\E\bigl[(W'-W)^2\,\bigl|\,X\bigr]\Bigr]
-\E\bigl[W^4\bigr]\\
&=3\sum_{\substack{M,N\subseteq[n]:\\\abs{M},\abs{N}\leq 2d}}a_M\E\bigl[U_MU_N\bigr]-\E\bigl[W^4\bigr]\,,
\end{align*}
where we recall that $a_M=1-\frac{\abs{M}}{2d}\in[0,1]$, for all $M\subseteq[n]$ such that $\abs{M}\leq 2d$. Noting that $a_\emptyset U_\emptyset^2=1$, $a_M=0$ whenever $\abs{M}=2d$ and using the orthogonality of the Hoeffding decomposition yield
\begin{align*}
 \frac{n}{4d}\E\bigl[(W'-W)^4\bigr]&=3a_\emptyset U_\emptyset^2-\E\bigl[W^4\bigr]+3\sum_{\substack{M\subseteq[n]:\\1\leq\abs{M}\leq 2d-1}}a_M\Var(U_M)\\
	&\leq 3-\E[W^4]+3\sum_{\substack{M\subseteq[n]:\\1\leq\abs{M}\leq 2d-1}}\Var(U_M)\\
	&\leq 3-\E[W^4]+3\bigl(\E[W^4]-3+\kappa_d \rho^2\bigr)\\
	&=2\bigl(\E[W^4]-3\bigr)+3\kappa_d \rho^2\,,
\end{align*}
where we have used Lemma \ref{varlemma} to obtain the last inequality.\\
\end{proof}

From the fact that 
\begin{equation*}
 \frac{2d}{n}=2\lambda=\E\bigl[(W'-W)^2\bigr]
\end{equation*}
and using the Cauchy-Schwarz inequality we obtain 
\begin{align}\label{t21}
 \frac{1}{3\lambda}\E\babs{W'-W}^3&\leq \frac{n}{3d}\Bigl(\E\bigl[(W'-W)^2\bigr]\Bigr)^{1/2}\Bigl(\E\babs{W'-W}^4\Bigr)^{1/2}\notag\\
 &=\frac{2\sqrt{2}}{3}\Bigl(\frac{n}{4d}\E\bigl[(W'-W)^4\bigr]\Bigr)^{1/2}\,.
\end{align}

Hence, by virtue of Lemma \ref{remlemma} we have 
\begin{equation}\label{t22}
 \frac{1}{3\lambda}\E\babs{W'-W}^3\leq \frac{2\sqrt{2}}{3} \Bigl(2\bigl(\E[W^4]-3\bigr)+3\kappa_d \rho^2\Bigr)^{1/2}\,.
\end{equation}

Theorem \ref{1dmt} now follows from \eqref{1dexbound}, Lemma \ref{1dterm1} and from \eqref{t22}\,.

\section{Proof of the multidimensional theorem}\label{mdim}

\subsection{Stein's method of exchangeable pairs for multivariate normal approximation}\label{mdimex}

Although the exchangeable pairs coupling lies at the heart of univariate normal approximation by Stein's method, it was only in 2008 in \cite{ChaMe08} that the problem of developing an analogous 
technique in the multivariate setting was finally attacked. 
In their work, for a given random vector 
$$W=(W(1),\ldots,W(r))^T,$$ the authors assume the existence of another random vector $$W'=(W'(1),\ldots,W'(r))^T,$$ defined 
on the same probability space $(\Omega,\F,\Prob)$, such that $W'$ has the same distribution as $W$ and such that the linear regression property 
\begin{equation*}
\E\bigl[W'-W\,\bigl|\, W\bigr]=-\lambda W 
\end{equation*}
is satisfied for some positive constant $\lambda$. Under these assumptions the authors prove several theorems which bound the distance 
from $W$ to a standard normal random vector in terms of the pair $(W,W')$.\\
In \cite{ReiRoe09} the authors motivate and investigate the more general linear regression property
\begin{equation}\label{linreggmn}
\E\bigl[W'-W\,\bigl|\, \G\bigr]=-\Lambda W+R\,,
\end{equation}
where now $\Lambda$ is an invertible non-random $r\times r$ matrix, $\G\subseteq\F$ is a sub-$\sigma$-field of $\F$ such that $\sigma(W)\subseteq\G$ and $R=(R(1),\ldots,R(r))^T$ 
is a small remainder term. However, in contrast to \cite{ChaMe08} and to the univariate situation presented in Subsection \ref{1dimex}, in \cite{ReiRoe09} the full strength of the exchangeability 
of the vector $(W,W')$ is needed. Finally, in \cite{Meck09} the two approaches from \cite{ChaMe08} and \cite{ReiRoe09} are combined, allowing for the more general linear regression property from \cite{ReiRoe09}
and using sharper coordinate-free bounds on the solution to the Stein equation similar to those derived in \cite{ChaMe08}. 
The following result, quoted from \cite{Doe12c}, is (a version of) Theorem 3 in \cite{Meck09} but with better constants.

\begin{theorem}\label{meckes}
Let $(W,W')$ be an exchangeable pair of $\R^r$-valued $L^2(\Prob)$ random vectors defined on a probability space $(\Omega,\F,\Prob)$ and let $\G\subseteq\F$ be 
a sub-$\sigma$-field of $\F$ such that $\sigma(W)\subseteq\G$. Suppose there exist a non-random invertible matrix $\Lambda\in\R^{r\times r}$, 
a non-random positive semidefinite matrix $\Sigma$, a $\G$-measurable random vector $R$ and a $\G$-measurable random matrix $S$ such that 
\eqref{linreggmn} and 
\begin{equation}\label{condrm1}
 \E\Bigl[(W'-W)(W'-W)^T\,\Bigl|\,\G\Bigr]=2\Lambda\Sigma +S
\end{equation}
hold true. Finally, denote by $Z$ a centered $r$-dimensional Gaussian vector with covariance matrix $\Sigma$.
\begin{enumerate}[{\normalfont (a)}]
 \item For any $h\in C^3(\R^r)$ such that $\E\bigl[\abs{h(W)}\bigr]<\infty$ and $\E\bigl[\abs{h(Z)}\bigr]<\infty$, 
\begin{align*}
\bigl|\E[h(W)]-\E[h(Z)]\bigr| 
&\leq\Opnorm{\Lambda^{-1}}\Biggl(M_1(h) \E\bigl[\Enorm{R}\bigr]+\frac{1}{4}\tilde{M}_2(h) \E\bigl[\HSnorm{S}\bigr]\\
&\;+\frac{1}{18}M_3(h)\E\bigl[\Enorm{W'-W}^3\bigr]\Biggr)\\
&\leq\Opnorm{\Lambda^{-1}}\Biggl(M_1(h) \E\bigl[\Enorm{R}\bigr]+\frac{\sqrt{r}}{4}M_2(h) \E\bigl[\HSnorm{S}\bigr]\\
&\;+\frac{1}{18}M_3(h)\E\bigl[\Enorm{W'-W}^3\bigr]\Biggr)\,.
\end{align*}

\item If $\Sigma$ is actually positive definite, then for each $h\in C^2(\R^r)$ such that\\ $\E\bigl[\abs{h(W)}\bigr]<\infty$ and $\E\bigl[\abs{h(Z)}\bigr]<\infty$ we have
\begin{align*}
 \bigl|\E[h(W)]-\E[h(Z)]\bigr| 
&\leq M_1(h)\Opnorm{\Lambda^{-1}}\Biggl(\E\bigl[\Enorm{R}\bigr] +\frac{\Opnorm{\Sigma^{-1/2}}}{\sqrt{2\pi}} \E\bigl[\HSnorm{S}\bigr]\Biggr)\\
&\quad+\frac{\sqrt{2\pi}}{24} M_2(h)\Opnorm{\Lambda^{-1}}\Opnorm{\Sigma^{-1/2}}\E\bigl[\Enorm{W'-W}^3\bigr]\,.
\end{align*}
\end{enumerate}
\end{theorem}

\subsection{Proof of Theorem \ref{mdmt}}\label{proofmd}
Recall the notation and assumptions from Subsection \ref{intro2}. Starting from the random vector $W=(W(1),\dotsc,W(r))^T$ we will construct another vector 
\begin{equation*}
W':=(W'(1),\dotsc,W'(r))^T
\end{equation*} 
such that $(W,W')$ is an exchangeable pair in the following way:
For each $1\leq i\leq r$ we construct $W'(i)$ in the same way as we did in the one-dimensional situation treated in Subsection \ref{proof1d} and from the same independent copy $Y=(Y_1,\dotsc,Y_n)$ of $X=(X_1,\dotsc,X_n)$ 
and the same $\alpha$ which is independent of $(X,Y)$ and uniformly distributed on $[n]$. We will apply Theorem \ref{meckes} with $\Sigma=\V$ and $\G=\sigma(X_1,\dotsc,X_n)$.

\begin{lemma}\label{linregmv}
With the above definitions and notation we have 
\begin{equation*}
 \E\bigl[W'-W\,\bigl|\,X\bigr]=-\Lambda W\,,
\end{equation*}
where the matrix $\Lambda$ is given by $\Lambda=\diag\Bigl(\frac{p_1}{n},\dotsc,\frac{p_r}{n}\Bigr)$.
\end{lemma}

\begin{proof}
 This follows immediately from Lemma \ref{linreg}.
\end{proof}

Hence, we obtain that 
\begin{equation}\label{oplainv}
 \Opnorm{\Lambda^{-1}}=\max_{i=1,\dotsc,r}\frac{n}{p_i}=\frac{n}{p_1}\,.
\end{equation}

Let us define the random matrix $S=(S_{i,k})_{1\leq i,k\leq r}$ by the relation
\begin{equation}\label{defs}
 \E\bigl[(W'-W)(W'-W)^T\,\bigl|\,X\bigr]=2\Lambda \V+S\,.
\end{equation}
From Lemma \ref{linregmv} and the fact that $v_{i,k}=0$ unless $p_i=p_k$ we easily conclude that $S$ is symmetric. 
Also, using exchangeability, it is readily checked that 
\begin{equation}\label{means}
\E\bigl[S\bigr]=\E\bigl[(W'-W)(W'-W)^T\bigr]-2\Lambda \V=0\,.
\end{equation}

\begin{lemma}\label{condexmult}
Let $1\leq i\leq k\leq r$ and let 
\[W(i)W(k)=\sum_{\substack{M\subseteq[n]:\\\abs{M}\leq p_i+p_k}}U_M(i,k)\]
be the Hoeffding decomposition of $W(i)W(k)$. Then, we have the Hoeffding decomposition  
\begin{equation*}
 n\E\bigl[(W'(i)-W(i))(W'(k)-W(k))\,\bigl|\,X\bigr]=\sum_{\substack{M\subseteq[n]:\\\abs{M}\leq p_i+p_k-1}}a_M(i,k)U_M(i,k)\,,
\end{equation*}
where 
\begin{equation*}
 a_M(i,k)=p_i+p_k-\abs{M}\,.
\end{equation*}
\end{lemma}

\begin{proof}
 First note that we have the representation 
 \begin{align*}
 &\quad (W'(i)-W(i))(W'(k)-W(k))\\
 &=\sum_{j=1}^n 1_{\{\alpha=j\}}\sum_{\substack{J\in\D_{p_i}, K\in\D_{p_k}:\\j\in J\cap K}}(W_J'(i)-W_J(i))(W_K'(k)-W_K(k))\\
 &=\sum_{j=1}^n 1_{\{\alpha=j\}}\sum_{\substack{J\in\D_{p_i}, K\in\D_{p_k}:\\j\in J\cap K}}\Bigl(W_J^{(j)}(i)W_K^{(j)}(k)+W_J(i)W_K(k)-W_J(i)W_K^{(j)}(k)\\
 &\hspace{5cm}-W_J^{(j)}(i)W_K(k)\Bigr)
 \end{align*}
which implies that 
\begin{align*}
 &n\E\bigl[(W'(i)-W(i))(W'(k)-W(k))\,\bigl|\,X\bigr]\\
 &=\sum_{j=1}^n\sum_{\substack{J\in\D_{p_i}, K\in\D_{p_k}:\\j\in J\cap K}}\Bigl(\E\bigl[W_J(i)W_K(k)\,\bigl|\,\F_{(J\cup K)\setminus\{j\}}\bigr]+W_J(i)W_K(k)\Bigr)\\
 &=\sum_{\substack{J\in\D_{p_i}, K\in\D_{p_k}:\\ J\cap K\not=\emptyset}}\abs{J\cap K}W_J(i)W_K(k)
 +\sum_{j=1}^n\sum_{\substack{J\in\D_{p_i}, K\in\D_{p_k}:\\j\in J\cap K}}\E\bigl[W_J(i)W_K(k)\,\bigl|\,\F_{(J\cup K)\setminus\{j\}}\bigr]
 \end{align*}

 Using Lemma \ref{hdle2} (b) we have 
\begin{align}\label{cem1}
 &\sum_{j=1}^n\sum_{\substack{J\in\D_{p_i}, K\in\D_{p_k}:\\j\in J\cap K}}\E\bigl[W_J(i)W_K(k)\,\bigl|\,\F_{(J\cup K)\setminus\{j\}}\bigr]\notag\\
 &=\sum_{j=1}^n\sum_{\substack{J\in\D_{p_i}, K\in\D_{p_k}:\\j\in J\cap K}}\sum_{\substack{M\subseteq[n]:\\J\Delta K\subseteq M\subseteq (J\cup K)\setminus\{j\}}}
 \sum_{\substack{L\subseteq[n]:\\J\Delta K\subseteq L\subseteq M}}(-1)^{\abs{M}-\abs{L}}\E\bigl[W_J(i)W_K(k)\,\bigl|\,\F_L\bigr]\notag\\
 &=\sum_{\substack{J\in\D_{p_i}, K\in\D_{p_k}:\\ J\cap K\not=\emptyset}}\sum_{j\in J\cap K}\sum_{\substack{M\subseteq[n]:\\J\Delta K\subseteq M\subseteq (J\cup K)\setminus\{j\}}}
 \sum_{\substack{L\subseteq[n]:\\J\Delta K\subseteq L\subseteq M}}(-1)^{\abs{M}-\abs{L}}\E\bigl[W_J(i)W_K(k)\,\bigl|\,\F_L\bigr]\notag\\
 &=\sum_{\substack{J\in\D_{p_i}, K\in\D_{p_k}:\\ J\cap K\not=\emptyset}}\sum_{\substack{M\subseteq[n]:\\J\Delta K\subseteq M\subsetneq J\cup K}}\bigl(\abs{J\cup K}-\abs{M}\bigr)
 \sum_{\substack{L\subseteq[n]:\\J\Delta K\subseteq L\subseteq M}}(-1)^{\abs{M}-\abs{L}}\E\bigl[W_J(i)W_K(k)\,\bigl|\,\F_L\bigr]\notag\\
 &=\sum_{\substack{J\in\D_{p_i}, K\in\D_{p_k}:\\ J\cap K\not=\emptyset}}\sum_{\substack{M\subseteq[n]:\\J\Delta K\subseteq M\subseteq J\cup K}}\bigl(\abs{J\cup K}-\abs{M}\bigr)
 \sum_{\substack{L\subseteq[n]:\\J\Delta K\subseteq L\subseteq M}}(-1)^{\abs{M}-\abs{L}}\E\bigl[W_J(i)W_K(k)\,\bigl|\,\F_L\bigr]
 \end{align}
 Note that for the third equality we have used the crucial fact that 
\begin{equation*}
 (J\cup K)\setminus M=(J\cap K)\setminus M\,,\quad\text{whenever } J\Delta K\subseteq M
\end{equation*}
which implies that 
\begin{equation*}
 \abs{(J\cap K)\setminus M}=\abs{J\cup K}-\abs{M}\quad\text{for }J\Delta K\subseteq M\subseteq J\cup K\,.
\end{equation*}
Also, from Lemma \ref{hdle2} (a) we obtain that 
\begin{align}\label{cem2}
&\sum_{\substack{J\in\D_{p_i}, K\in\D_{p_k}:\\ J\cap K\not=\emptyset}}\abs{J\cap K}\sum_{\substack{M\subseteq[n]:\\J\Delta K\subseteq M\subseteq J\cup K}}
 \sum_{\substack{L\subseteq[n]:\\J\Delta K\subseteq L\subseteq M}}(-1)^{\abs{M}-\abs{L}}\E\bigl[W_J(i)W_K(k)\,\bigl|\,\F_L\bigr]\notag\\ 
 &=\sum_{\substack{J\in\D_{p_i}, K\in\D_{p_k}:\\ J\cap K\not=\emptyset}}\abs{J\cap K}\sum_{\substack{M\subseteq[n]:\\J\Delta K\subseteq M\subseteq J\cup K}}
 \sum_{\substack{L\subseteq[n]:\\J\Delta K\subseteq L\subseteq M}}(-1)^{\abs{M}-\abs{L}}\E\bigl[W_J(i)W_K(k)\,\bigl|\,\F_L\bigr]\notag\\
 &=\sum_{\substack{J\in\D_{p_i}, K\in\D_{p_k}:\\ J\cap K\not=\emptyset}}\sum_{\substack{M\subseteq[n]:\\J\Delta K\subseteq M\subseteq J\cup K}}\abs{J\cap K}
 \sum_{\substack{L\subseteq[n]:\\J\Delta K\subseteq L\subseteq M}}(-1)^{\abs{M}-\abs{L}}\E\bigl[W_J(i)W_K(k)\,\bigl|\,\F_L\bigr]
\end{align}

Combining \eqref{cem1} and \eqref{cem2} we thus have 
\begin{align*}
 &n\E\bigl[(W'(i)-W(i))(W'(k)-W(k))\,\bigl|\,X\bigr]\\
 &=\sum_{\substack{J\in\D_{p_i}, K\in\D_{p_k}:\\ J\cap K\not=\emptyset}}\sum_{\substack{M\subseteq[n]:\\J\Delta K\subseteq M\subseteq J\cup K}}\bigl(\abs{J\cup K}+\abs{J\cap K}-\abs{M}\bigr)\\
 &\hspace{3cm}\sum_{\substack{L\subseteq[n]:\\J\Delta K\subseteq L\subseteq M}}(-1)^{\abs{M}-\abs{L}}\E\bigl[W_J(i)W_K(k)\,\bigl|\,\F_L\bigr]\\
 &=\sum_{\substack{J\in\D_{p_i}, K\in\D_{p_k}:\\ J\cap K\not=\emptyset}}\sum_{\substack{M\subseteq[n]:\\J\Delta K\subseteq M\subseteq J\cup K}}\bigl(p_i+p_k-\abs{M}\bigr)
 \sum_{\substack{L\subseteq[n]:\\J\Delta K\subseteq L\subseteq M}}(-1)^{\abs{M}-\abs{L}}\E\bigl[W_J(i)W_K(k)\,\bigl|\,\F_L\bigr]\\
 &=\sum_{\substack{M\subseteq[n]:\\\abs{M}\leq p_i+p_k-1}}\bigl(p_i+p_k-\abs{M}\bigr)\sum_{\substack{J\in\D_{p_i}, K\in\D_{p_k}:\\ J\cap K\not=\emptyset,\\J\Delta K\subseteq M\subseteq J\cup K }}
 \sum_{\substack{L\subseteq[n]:\\J\Delta K\subseteq L\subseteq M}}(-1)^{\abs{M}-\abs{L}}\E\bigl[W_J(i)W_K(k)\,\bigl|\,\F_L\bigr]\\
 &=\sum_{\substack{M\subseteq[n]:\\\abs{M}\leq p_i+p_k-1}}\bigl(p_i+p_k-\abs{M}\bigr)\sum_{\substack{J\in\D_{p_i}, K\in\D_{p_k}:\\ J\Delta K\subseteq M\subseteq J\cup K }}
 \sum_{\substack{L\subseteq[n]:\\J\Delta K\subseteq L\subseteq M}}(-1)^{\abs{M}-\abs{L}}\E\bigl[W_J(i)W_K(k)\,\bigl|\,\F_L\bigr]\\
 &=\sum_{\substack{M\subseteq[n]:\\\abs{M}\leq p_i+p_k-1}}\bigl(p_i+p_k-\abs{M}\bigr)U_M(i,k)\,,
\end{align*}
as claimed. 
\end{proof}

Since $S$ is centered, from \eqref{defs} and Lemma \ref{condexmult} we obtain that 
\begin{align}\label{mv1}
 n^2\E\bigl[S_{i,k}^2\bigr]&=\Var\bigl(nS_{i,k}\bigr)=\Var\Bigl(n\E\bigl[(W'(i)-W(i))(W'(k)-W(k))\,\bigl|\,X\bigr]\Bigr)\notag\\
 &=\sum_{\substack{M\subseteq[n]:\\\abs{M}\leq p_i+p_k-1}}\bigl(p_i+p_k-\abs{M}\bigr)^2\Var\bigl(U_M(i,k)\bigr)\notag\\
 &\leq(p_i+p_k)^2\sum_{\substack{M\subseteq[n]:\\\abs{M}\leq p_i+p_k-1}}\Var\bigl(U_M(i,k)\bigr)\notag\\
 &=(p_i+p_k)^2\Bigl(\Var(W(i)W(k))-\sum_{\substack{M\subseteq[n]:\\\abs{M}= p_i+p_k}}\E\bigl[U_M(i,k)^2\bigr]\Bigr)\notag\\
 &=(p_i+p_k)^2\Bigl(\Var(W(i)W(k))\notag\\
 &\hspace{3cm}-\sum_{\substack{J,L\in\D_{p_i},K,M\in\D_{p_k}:\\J\cap K=L\cap M=\emptyset}}\E\bigl[W_J(i)W_K(k)W_L(i)W_M(k)\bigr]\Bigr)\,.
 \end{align}
  
For $1\leq i\leq k\leq r$ define 
\begin{equation*}
 S_0(i,k):=\sum_{\substack{J,L\in\D_{p_i},K,M\in\D_{p_k}:\\J\cap K=L\cap M=\emptyset,\\\emptyset\subsetneq J\cap L=J\setminus(J\cap M)\subsetneq J,\\\emptyset\subsetneq L\cap J=L\setminus(L\cap K)\subsetneq L}}
 \E\bigl[W_J(i)W_K(k)W_L(i)W_M(k)\bigr]\,.
\end{equation*}
If $p_i<p_k$, then from \eqref{mv1} we have that 
\begin{align}\label{mv1a}
 n^2\E\bigl[S_{i,k}^2\bigr]&\leq (p_i+p_k)^2\Bigl(\Var(W(i)W(k))-\sum_{\substack{J\in\D_{p_i},K\in\D_{p_k}:\\ J\cap K=\emptyset}}\sigma_{J}(i)^2\sigma_{K}(k)^2-S_0(i,k)\Bigr)\,.
\end{align}

Lemma \ref{sigmale} immediately yields that 
\begin{align}\label{mv3}
 \sum_{\substack{J\in\D_{p_i},K\in\D_{p_k}:\\ J\cap K=\emptyset}}\sigma_{J}(i)^2\sigma_{K}(k)^2&=1-\sum_{\substack{J\in\D_{p_i},K\in\D_{p_k}:\\ J\cap K\not=\emptyset}}\sigma_{J}(i)^2\sigma_{K}(k)^2\notag\\
 &\geq 1-\min\Bigl(p_i \rho_{n,k}^2\,,\,p_k \rho_{n,i}^2\Bigr)\,.
\end{align}
If $p_i=p_k$, then we obtain that
\begin{align}\label{mv1b}
 n^2\E\bigl[S_{i,k}^2\bigr]&\leq 4p_i^2\Bigl(\Var(W(i)W(k))-\sum_{\substack{J,K\in\D_{p_i}:\\ J\cap K=\emptyset}}\sigma_{J}(i)^2\sigma_{K}(k)^2\notag\\
 &\quad -\sum_{\substack{J,K\in\D_{p_i}:\\ J\cap K=\emptyset}}\E\bigl[W_J(i)W_J(k)\bigr]\E\bigl[W_K(i)W_K(k)\bigr]-S_0(i,k)\Bigr)\,.
\end{align}
Similarly to Lemma \ref{sigmale} we obtain for $p_i=p_k$ that
\begin{align*}
 &\Babs{\sum_{\substack{J,K\in\D_{p_i}:\\ J\cap K\not=\emptyset}}\E\bigl[W_J(i)W_J(k)\bigr]\E\bigl[W_K(i)W_K(k)\bigr]}\notag\\
 &\leq\sum_{J\in\D_{p_i}}\babs{\E\bigl[W_J(i)W_J(k)\bigr]}\sum_{j\in J}\sum_{\substack{K\in\D_{p_i}:\\j\in K}}\babs{\E\bigl[W_K(i)W_K(k)\bigr]}\notag\\
 &\leq\sum_{J\in\D_{p_i}}\babs{\E\bigl[W_J(i)W_J(k)\bigr]}\sum_{j\in J}\Bigl(\sum_{\substack{K\in\D_{p_i}:\\j\in K}}\sigma_K(i)^2\Bigr)^{1/2}\Bigl(\sum_{\substack{K\in\D_{p_k}:\\j\in K}}\sigma_K(k)^2\Bigr)^{1/2}\notag\\
 &\leq p_i\rho_{n,k}\rho_{n,i}\sum_{J\in\D_{p_i}}\babs{\E\bigl[W_J(i)W_J(k)\bigr]}\leq p_i\rho_{n,k}\rho_{n,i}\sum_{J\in\D_{p_i}}\sigma_J(i)\sigma_J(k)\notag\\
 &\leq p_i\rho_{n,k}\rho_{n,i}\Bigl(\sum_{J\in\D_{p_i}}\sigma_J(i)^2\Bigr)^{1/2}\Bigl(\sum_{J\in\D_{p_k}}\sigma_J(k)^2\Bigr)^{1/2}\notag\\
 &=p_i\rho_{n,k}\rho_{n,i}\,.
\end{align*}
Hence, if $p_i=p_k$ we have that 
\begin{align}\label{mv7}
 &\sum_{\substack{J,K\in\D_{p_i}:\\ J\cap K=\emptyset}}\E\bigl[W_J(i)W_J(k)\bigr]\E\bigl[W_K(i)W_K(k)\bigr]\notag\\
& =\Bigl(\sum_{J\in\D_{p_i}}\E\bigl[W_J(i)W_J(k)\bigr]\Bigr)^2-\sum_{\substack{J,K\in\D_{p_i}:\\ J\cap K\not=\emptyset}}\E\bigl[W_J(i)W_J(k)\bigr]\E\bigl[W_K(i)W_K(k)\bigr]\notag\\
&=v_{i,k}^2-\sum_{\substack{J,K\in\D_{p_i}:\\ J\cap K\not=\emptyset}}\E\bigl[W_J(i)W_J(k)\bigr]\E\bigl[W_K(i)W_K(k)\bigr]\notag\\
&\geq v_{i,k}^2-p_i\rho_{n,k}\rho_{n,i}\,.
\end{align}

Note that we can write
\begin{align}\label{mv6}
 \Var\bigl(W(i)W(k)\bigr)&=\E\bigl[W(i)^2W(k)^2\bigr]-\Bigl(\E\bigl[W(i)W(k)\bigr]\Bigr)^2\notag\\
 &=\Cov\bigl(W(i)^2,W(k)^2\bigr)+\E\bigl[W(i)^2\bigr]\E\bigl[W(k)^2\bigr]-v_{i,k}^2\notag\\
 &=\Cov\bigl(W(i)^2,W(k)^2\bigr)+1-v_{i,k}^2\,.
\end{align}

Hence, if $p_i<p_k$, then, since $v_{i,k}=0$, from \eqref{mv1a}, \eqref{mv3} and \eqref{mv6} we see that 
\begin{align}\label{mv4a}
 n^2\E\bigl[S_{i,k}^2\bigr]&\leq (p_i+p_k)^2\Bigl(\Cov\bigl(W(i)^2,W(k)^2\bigr)+\min\Bigl(p_i \rho_{n,k}^2\,,\,p_k \rho_{n,i}^2\Bigr)-S_0(i,k)\Bigr)\,.
\end{align}
If, on the other hand, $p_i=p_k$, then from \eqref{mv1b}, \eqref{mv3}, \eqref{mv6} and \eqref{mv7} we conclude that
\begin{align}\label{mv4b}
 n^2\E\bigl[S_{i,k}^2\bigr]&\leq 4p_i^2\Bigl(\Cov\bigl(W(i)^2,W(k)^2\bigr)-2v_{i,k}^2+p_i\min\Bigl(\rho_{n,k}^2\,,\,\rho_{n,i}^2\Bigr)\notag\\
 &\hspace{3cm}+p_i\rho_{n,k}\rho_{n,i}-S_0(i,k)\Bigr)\notag\\
 &=4p_i^2\Bigl(\E\bigl[W(i)^2W(k)^2\bigr]-1-2v_{i,k}^2+p_i\min\Bigl(\rho_{n,k}^2\,,\,\rho_{n,i}^2\Bigr)\notag\\
 &\hspace{3cm}+p_i\rho_{n,k}\rho_{n,i}-S_0(i,k)\Bigr)\notag\\
 &=4p_i^2\Bigl(\E\bigl[W(i)^2W(k)^2\bigr]-\E\bigl[Z(i)^2Z(k)^2\bigr]+p_i\min\Bigl(\rho_{n,k}^2\,,\,\rho_{n,i}^2\Bigr)\notag\\
 &\hspace{3cm}+p_i\rho_{n,k}\rho_{n,i}-S_0(i,k)\Bigr)\,.
\end{align}
For the last identity we have used the elementarily verifiable fact that 
\begin{equation*}
 \E\bigl[Z(i)^2Z(k)^2\bigr]=1+2v_{i,k}^2
\end{equation*}
for all $1\leq i,k\leq r$.\\
For $p_i< p_k$, by the orthogonality of the Hoeffding decomposition and by the Cauchy-Schwarz inequality we have that 
\begin{align}\label{mv5}
&\Cov\bigl(W(i)^2,W(k)^2\bigr)=\sum_{\substack{M,N\subseteq[n]:\\\abs{M}\leq2 p_i,\abs{N}\leq2p_k}}\E\bigl[U_M(i,i)U_N(k,k)\bigr]-\E[W(i)^2]\E[W(k)^2]\notag\\
&=\sum_{\substack{M\subseteq[n]:\\1\leq\abs{M}\leq2 p_i}}\E\bigl[U_M(i,i)U_M(k,k)\bigr]\notag\\
&\leq\sum_{\substack{M\subseteq[n]:\\1\leq\abs{M}\leq2 p_i}}\Bigl(\E\bigl[U_M(i,i)^2\bigr]\Bigr)^{1/2}\Bigl(\E\bigl[U_M(k,k)^2\bigr]\Bigr)^{1/2}\notag\\
&\leq\biggl(\sum_{\substack{M\subseteq[n]:\\1\leq\abs{M}\leq2 p_i}}\E\bigl[U_M(i,i)^2\bigr]\biggr)^{1/2}\biggl(\sum_{\substack{M\subseteq[n]:\\1\leq\abs{M}\leq2 p_i}}\E\bigl[U_M(k,k)^2\bigr]\biggr)^{1/2}\notag\\
&=\biggl(\sum_{\substack{M\subseteq[n]:\\1\leq\abs{M}\leq2 p_i}}\Var\bigl(U_M(i,i)\bigr)\biggr)^{1/2}\biggl(\sum_{\substack{M\subseteq[n]:\\1\leq\abs{M}\leq2 p_i}}\Var\bigl(U_M(k,k)\bigr)\biggr)^{1/2}\notag\\
&=\Bigl(\E\bigl[W(i)^4\bigr]-1\Bigr)^{1/2}\biggl(\sum_{\substack{M\subseteq[n]:\\1\leq\abs{M}\leq2 p_i}}\Var\bigl(U_M(k,k)\bigr)\biggr)^{1/2}\,.
\end{align}
Since $p_i<p_k$, by means of \eqref{mv5} we can further bound 
\begin{align}\label{mv8}
 &\Cov\bigl(W(i)^2,W(k)^2\bigr)\leq\Bigl(\E\bigl[W(i)^4\bigr]-1\Bigr)^{1/2}\biggl(\sum_{\substack{M\subseteq[n]:\\1\leq\abs{M}\leq2 p_k-1}}\Var\bigl(U_M(k,k)\bigr)\biggr)^{1/2}\notag\\
&=\Bigl(\E\bigl[W(i)^4\bigr]-1\Bigr)^{1/2}\Bigl(\Var\bigl(W(k)^2\bigr)-\sum_{\substack{M\subseteq[n]:\\\abs{M}=2 p_k}}\E\bigl[U_M(k,k)^2\bigr]\Bigr)^{1/2}\notag\\
&\leq \Bigl(\E\bigl[W(i)^4\bigr]-1\Bigr)^{1/2}\Bigl(\E\bigl[W(k)^4\bigr]-3+\bigl(2p_k+C_{p_k}\bigr)\rho_{n,k}^2\Bigr)^{1/2}\,,
\end{align}
where the final inequality is true by \eqref{compt1}.

From \eqref{mv4a} and \eqref{mv8} and from \eqref{mv4b}, respectively, we thus obtain the following result. 
\begin{lemma}\label{varmult}
Let $1\leq i\leq k\leq r$.
\begin{enumerate}[{\normalfont(i)}]
\item If $p_i<p_k$, then 
\begin{align*}
 n^2\E\bigl[S_{i,k}^2\bigr]&\leq (p_i+p_k)^2\Biggl[\Bigl(\E\bigl[W(i)^4\bigr]-1\Bigr)^{1/2}\Bigl(\E\bigl[W(k)^4\bigr]-3+\bigl(2p_k+C_{p_k}\bigr)\rho_{n,k}^2\Bigr)^{1/2}\\
 &\hspace{3cm}+\min\Bigl(p_i \rho_{n,k}^2\,,\,p_k \rho_{n,i}^2\Bigr)-S_0(i,k)\Biggr]\,.
\end{align*}
\item If $p_i=p_k$, then 
\begin{align*}
 n^2\E\bigl[S_{i,k}^2\bigr]&\leq4p_i^2\Bigl(\E\bigl[W(i)^2W(k)^2\bigr]-\E\bigl[Z(i)^2Z(k)^2\bigr]+p_i\min\Bigl(\rho_{n,k}^2\,,\,\rho_{n,i}^2\Bigr)\notag\\
 &\hspace{3cm}+p_i\rho_{n,k}\rho_{n,i}-S_0(i,k)\Bigr)\,.
\end{align*}
\end{enumerate}
\end{lemma}

It remains to bound the quantities $S_0(i,k)$, $1\leq i\leq k\leq r$. The concepts of free indices and bifold quadruples from Subsection \ref{proof1d} generalize in the obvious way to quadruples 
$(J_1,J_2,J_3,J_4)\in\D_{p_i}\times\D_{p_k}\times\D_{p_i}\times\D_{p_k}=:\D_{i,k}^4$. We denote by $\B_{i,k}$ the collection of all bifold quadruples in $\D_{i,k}^4$.
Also, we denote by $\mathcal{T}_{i,k}$ the set of quadruples $(J_1,J_2,J_3,J_4)\in\D_{i,k}^4$ which are neither bifold 
nor have a free index, i.e. which satisfy 
\begin{equation*}
 1_{J_1}+1_{J_2}+1_{J_3}+1_{J_4}\geq 21_{J_1\cup J_2\cup J_3\cup J_4}
\end{equation*}
and there is a $j\in[n]$ such that 
\begin{equation*}
 1_{J_1}(j)+1_{J_2}(j)+1_{J_3}(j)+1_{J_4}(j)\geq3\,. 
\end{equation*}
With these definitions, for $1\leq i\leq k\leq r$, we define 
\begin{equation*}
 \tau_{i,k}:=\sum_{(J,K,L,M)\in\mathcal{T}_{i,k}}\sigma_{J}(i)\sigma_K(k)\sigma_L(i)\sigma_M(k)\,.
\end{equation*}

The next result is a generalization of Proposition \ref{s0prop}.
\begin{prop}\label{s0propgen}
 With these definitions, for $1\leq i\leq k\leq r$, we have 
 \begin{equation*}
  S_0(i,k)\geq-\tau_{i,k}\,.
 \end{equation*}
\end{prop}

The proof is postponed to Section \ref{proofs}.
It remains to obtain a bound on the quantities $\tau_{i,k}$ in terms of $\rho_{n,i}^2$ and $\rho_{n,k}^2$. This is provided by the following result which generalizes Proposition \ref{taubound}. 
An outline of the main elements of the proof is given in Section \ref{proofs}.

\begin{prop}\label{gentaubound}
For each $1\leq i, k\leq r$, there exists a finite constant $C_{i,k}$ which depends on $i$ and $k$ only through $p_i$ and $p_k$ and which is independent of $n$ such that 
\[\tau_{i,k}=\sum_{(J,K,L,M)\in\mathcal{T}_{i,k}}\sigma_{J}(i)\sigma_K(k)\sigma_L(i)\sigma_M(k)\leq C_{i,k}\max\bigl(\rho_{n,i}^2,\rho_{n,k}^2\bigr)\,.\]
Furthermore, we have $C_{i,k}=C_{k,i}$.
\end{prop}

Combining Propositions \ref{s0propgen} and \ref{gentaubound}, we thus obtain that 
\begin{equation}\label{s0boundgen}
 S_0(i,k)\geq-C_{i,k}\max\bigl(\rho_{n,i}^2,\rho_{n,k}^2\bigr)
\end{equation}
for all $1\leq i,k\leq r$.\\

Observe that, using \eqref{oplainv} and the symmetry of $S$, we can bound
\begin{equation}\label{S1}
 \Opnorm{\Lambda^{-1}}\E\bigl[\HSnorm{S}\bigr]\leq \frac{n}{p_1}\Bigl(\E\bigl[\HSnorm{S}^2\bigr]\Bigr)^{1/2}
 =\frac{1}{p_1}\Bigl(\sum_{i,k=1}^r n^2\E\bigl[S_{i,k}^2\bigr]\Bigr)^{1/2}\,.
\end{equation}
Now, using Lemma \ref{varmult} we have 
\begin{align}\label{S2}
 \sum_{i,k=1}^r n^2\E\bigl[S_{i,k}^2\bigr]&=\sum_{l,m=1}^s\sum_{i=r_{l-1}+1}^{r_l}\sum_{k=r_{m-1}+1}^{r_m}n^2\E\bigl[S_{i,k}^2\bigr]\notag\\
&=\sum_{l=1}^s\sum_{i,k=r_{l-1}+1}^{r_l}n^2\E\bigl[S_{i,k}^2\bigr]+2\sum_{1\leq l<m\leq s}\sum_{i=r_{l-1}+1}^{r_l}\sum_{k=r_{m-1}+1}^{r_m}n^2\E\bigl[S_{i,k}^2\bigr]\notag\\
&\leq\sum_{l=1}^s4q_l^2\sum_{i,k=r_{l-1}+1}^{r_l}\Bigl(\E\bigl[W(i)^2W(k)^2\bigr]-\E\bigl[Z(i)^2Z(k)^2\bigr]\notag\\
&\hspace{2cm}+q_l\min\Bigl(\rho_{n,k}^2\,,\,\rho_{n,i}^2\Bigr)+q_l\rho_{n,k}\rho_{n,i}+C_{i,k}\max\bigl(\rho_{n,i}^2,\rho_{n,k}^2\bigr)\Bigr)\notag\\
&\;+2\sum_{1\leq l<m\leq s}(q_l+q_m)^2\sum_{i=r_{l-1}+1}^{r_l}\sum_{k=r_{m-1}+1}^{r_m}\Biggl[\Bigl(\E\bigl[W(i)^4\bigr]-1\Bigr)^{1/2}\notag\\
&\hspace{2cm}\Bigl(\E\bigl[W(k)^4\bigr]-3+\bigl(2q_m+C_{q_m}\bigr)\rho_{n,k}^2\Bigr)^{1/2}\notag\\
 &\hspace{3cm}+\min\Bigl(q_l \rho_{n,k}^2\,,\,q_m \rho_{n,i}^2\Bigr)+C_{i,k}\max\bigl(\rho_{n,i}^2,\rho_{n,k}^2\bigr)\Biggr]\notag\\
 &=A\,.
\end{align}
Here, the constants $C_{q_m}$ are defined by Proposition \ref{taubound}.

Note that from Lemma \ref{remlemma} applied to the exchangeable pair $(W(i),W'(i))$ we have 
\begin{equation}\label{R2}
 \frac{n}{4p_i}\E\babs{W'(i)-W(i)}^4\leq 2\bigl(\E\bigl[W(i)^4\bigr]-3\bigr)+ 3\bigl(C_{p_i}+2p_i\bigr)\rho_{n,i}^2\,.
\end{equation}

Using \eqref{oplainv} as well as Jensen's inequality, we obtain
\begin{align*}
 & \Opnorm{\Lambda^{-1}}\E\bigl[\Enorm{W'-W}^3\bigr]=\frac{n}{p_1}\E\bigl[\Enorm{W'-W}^3\bigr]\notag\\
 &=\frac{n}{p_1}\Bigl(\sum_{i=1}^r\E\babs{W'(i)-W(i)}^2\Bigr)^{3/2}=\frac{n}{p_1}\Bigl(r\sum_{i=1}^r\E\babs{W'(i)-W(i)}^2\frac{1}{r}\Bigr)^{3/2}\notag\\
& \leq\frac{n}{p_1}r^{3/2}\sum_{i=1}^r\E\babs{W'(i)-W(i)}^3\frac{1}{r}\notag\\
& =\frac{n}{p_1}r^{1/2}\sum_{i=1}^r\E\babs{W'(i)-W(i)}^3\,.
\end{align*}
Thus, by \eqref{R2} we have
\begin{align}\label{R1}
&\Opnorm{\Lambda^{-1}}\E\bigl[\Enorm{W'-W}^3\bigr]\notag\\
&\leq\frac{n}{p_1}r^{1/2}\sum_{i=1}^r\Bigl(\E\babs{W'(i)-W(i)}^2\Bigr)^{1/2}\Bigl(\E\babs{W'(i)-W(i)}^4\Bigr)^{1/2}\notag\\
&=2\sqrt{2r}\sum_{i=1}^r\frac{p_i}{p_1}\Bigl(\frac{n}{4p_i}\E\babs{W'(i)-W(i)}^4\Bigr)^{1/2}\notag\\
&\leq2\sqrt{2r}\sum_{i=1}^r\frac{p_i}{p_1}\Bigl(2\bigl(\E\bigl[W(i)^4\bigr]-3\bigr)+ 3\bigl(C_{p_i}+2p_i\bigr)\rho_{n,i}^2\Bigr)^{1/2}\notag\\
&=2\sqrt{2r}\sum_{l=1}^s\frac{q_l}{q_1}\sum_{i=r_{l-1}+1}^{r_l}\Bigl(2\bigl(\E\bigl[W(i)^4\bigr]-3\bigr)+ 3\bigl(C_{q_l}+2q_l\bigr)\rho_{n,i}^2\Bigr)^{1/2}\,.
\end{align}

Theorem \ref{mdmt} now follows from Theorem \ref{meckes} and from the respective bounds \eqref{S1}, \eqref{S2} and \eqref{R1}\,.


\section{Proofs of several technical results}\label{proofs}

\begin{lemma}\label{sigmale}
 In the situation of Section \ref{mdim}, for all $1\leq i\leq k\leq r$ we have 
 \begin{equation*}
 \sum_{\substack{J\in\D_{p_i},K\in\D_{p_k}:\\ J\cap K\not=\emptyset}}\sigma_{J}(i)^2\sigma_{K}(k)^2\leq\min\Bigl(p_i 
\rho_{n,k}^2\,,\,p_k \rho_{n,i}^2\Bigr)\,.
\end{equation*}
\end{lemma}

\begin{proof}
 Note that we have 
\begin{align*}
 \sum_{\substack{J\in\D_{p_i},K\in\D_{p_k}:\\ J\cap K\not=\emptyset}}\sigma_{J}(i)^2\sigma_{K}(k)^2
 &=\sum_{J\in\D_{p_i}}\sigma_{J}(i)^2\sum_{\substack{K\in\D_{p_k}:\\J\cap K\not=\emptyset}}\sigma_{K}(k)^2\\
 &\leq\sum_{J\in\D_{p_i}}\sigma_{J}(i)^2\sum_{j\in J}\sum_{\substack{K\in\D_{p_k}:\\j\in K}}\sigma_{K}(k)^2\\
 &\leq p_i \rho_{k}^2\sum_{J\in\D_{p_i}}\sigma_{J}(i)^2=p_i \rho_{k}^2\,.
\end{align*}
The claim follows by symmetry.
\end{proof}

\begin{lemma}[Generalization of Lemma 4 of \cite{deJo90}]\label{le4dj}
 Let $(J,K,L,M)\in \B_{i,k}$ be a bifold quadruple. Then, in the situation of Section \ref{mdim} we have 
 \begin{align*}
  \E\bigl[W_JW_KW_LW_M\bigr]&=\E\Bigl[\E\bigl[W_JW_K\,\bigl|\,\F_{J\Delta K}\bigr]\E\bigl[W_LW_M\,\bigl|\,\F_{L\Delta M}\bigr]\Bigr]\,.
 \end{align*}
 \end{lemma}
\begin{proof}
We repeat the short proof from \cite{deJo90}.  By independence, we have 
\begin{align*}
\E\bigl[W_JW_KW_LW_M\bigr]&=\E\Bigl[W_JW_K \E\bigl[W_LW_M\,\bigl|\,\F_{L\cup M}\bigr]\Bigr]\\
&=\E\Bigl[W_JW_K \E\bigl[W_LW_M\,\bigl|\,\F_{(J\cup K)\cap(L\cup M)}\bigr]\Bigr]\\
&=\E\Bigl[\E\bigl[W_JW_K\,\bigl|\,\F_{(J\cup K)\cap(L\cup M)}\bigr] \E\bigl[W_LW_M\,\bigl|\,\F_{(J\cup K)\cap(L\cup M)}\bigr]\Bigr]\,.
\end{align*}
Now, the claim follows from the fact that for a bifold quadruple $(J,K,L,M)$ the identity
\begin{equation*}
 (J\cup K)\cap(L\cup M)=J\Delta K=L\Delta M
\end{equation*}
holds true.
\end{proof}

\begin{lemma}[Generalization of Lemma 3 of \cite{deJo90}]\label{le3dj}
In the situation of Section \ref{mdim}, for $1\leq i\leq k\leq r$, $J\in\D_{p_i}$ and $K\in\D_{p_k}$ we have
\begin{equation*}
\E\Bigl[\Bigl(\E\bigl[W_JW_K\,\bigl|\,\F_{J\Delta K}\bigr]\Bigr)^2\Bigr]\leq\sigma_J^2(i)\sigma_K^2(k)\,.
\end{equation*}
\end{lemma}

\begin{proof}
 Again, we immitate the proof given in \cite{deJo90}. Using first the conditional version of the Cauchy-Schwarz inequality and then twice the independence of the underlying random variables $X_1,\dotsc,X_n$ we obtain 
 \begin{align*}
 \E\Bigl[\Bigl(\E\bigl[W_JW_K\,\bigl|\,\F_{J\Delta K}\bigr]\Bigr)^2\Bigr]&\leq\E\Bigl[ \E\bigl[W_J^2\,\bigl|\,\F_{J\Delta K}\bigr]\E\bigl[W_K^2\,\bigl|\,\F_{J\Delta K}\bigr]\Bigr]\\
 &=\E\Bigl[ \E\bigl[W_J^2\,\bigl|\,\F_{J\setminus K}\bigr]\E\bigl[W_K^2\,\bigl|\,\F_{K\setminus J}\bigr]\Bigr]\\
 &=\E\Bigl[ \E\bigl[W_J^2\,\bigl|\,\F_{J\setminus K}\bigr]\Bigr]\E\Bigl[\E\bigl[W_K^2\,\bigl|\,\F_{K\setminus J}\bigr]\Bigr]\\
 &=\sigma_J^2(i)\sigma_K^2(k)\,.
 \end{align*}
\end{proof}

\begin{proof}[Proof of Proposition \ref{s0propgen}]
 We generalize the argument used in the proof of Proposition 5 (b) of \cite{deJo90}. For $l=1,\dotsc, p_i-1$ we have 
 \begin{align*}
0&\leq \sum_{\substack{C\subseteq[n]:\\\abs{C}=p_i+p_k-2l}}\sum_{\substack{B,B'\subseteq C:\\\abs{B}=p_i-l,\abs{B'}=p_k-l,\\B\cap B'=\emptyset}}
\E\Bigl[\Bigl(\sum_{\substack{J\in\D_{p_i},M\in\D_{p_k}:\\J\setminus M= B, M\setminus J=B'}}\E\bigl[W_J(i)W_M(k)\,\bigl|\,\F_{J\Delta M}\bigr]\Bigl)^2\Bigr]\\
&=\sum_{\substack{C\subseteq[n]:\\\abs{C}=p_i+p_k-2l}}\sum_{\substack{B,B'\subseteq C:\\\abs{B}=p_i-l,\abs{B'}=p_k-l,\\B\cap B'=\emptyset}}
\sum_{\substack{(J,K,L,M)\in\D_{i,k}^4:\\J\setminus M= L\setminus K=B, M\setminus J=K\setminus L=B'}}\\
&\hspace{2cm}\E\Bigl[\E\bigl[W_J(i)W_M(k)\,\bigl|\,\F_{J\Delta M}\bigr]\E\bigl[W_L(i)W_K(k)\,\bigl|\,\F_{L\Delta K}\bigr]\Bigr]\\
&=\sum_{\substack{(J,K,L,M)\in\B_{i,k}:\\J\setminus M= L\setminus K, M\setminus J=K\setminus L,\\\abs{J\cap M}=l}}\E\bigl[W_J(i)W_K(k)W_L(i)W_M(k)\bigr]\\
&\;+\sum_{\substack{(J,K,L,M)\in\mathcal{T}_{i,k}:\\J\setminus M= L\setminus K, M\setminus J=K\setminus L,\\\abs{J\cap M}=l=\abs{L\cap K}}}
\E\Bigl[\E\bigl[W_J(i)W_M(k)\,\bigl|\,\F_{J\Delta M}\bigr]\E\bigl[W_L(i)W_K(k)\,\bigl|\,\F_{L\Delta K}\bigr]\Bigr]\\
&=:S_0(i,k,l)+R_l\,,
 \end{align*}
 where we have used Lemma \ref{le4dj} to obtain the second equality. Note that 
 \[\sum_{l=1}^{p_i-1}S_0(i,k,l)=S_0(i,k)\]
 because for a bifold quadruple $(J,K,L,M)$ the identity $J\setminus M =L\setminus K$ implies 
 that $J\cap K=L\cap M=\emptyset$ and because we have
 \begin{equation*}
  S_0(i,k)=\sum_{\substack{(J,K,L,M)\in\B_{i,k}:\\ J\cap K=L\cap M=\emptyset,\\\emptyset\subsetneq J\cap M\subsetneq J}}\E\bigl[W_J(i)W_K(k)W_L(i)W_M(k)\bigr]\,.
 \end{equation*}

Further, by the Cauchy-Schwarz inequality and by Lemma \ref{le3dj}, we have
 \begin{align*}
  \Babs{\sum_{l=1}^{p_i-1}R_l}&\leq \sum_{l=1}^{p_i-1}\sum_{\substack{(J,K,L,M)\in\mathcal{T}_{i,k}:\\J\setminus M= L\setminus K, M\setminus J=K\setminus L,\\\abs{J\cap M}=l=\abs{L\cap K}}}
  \biggl(\E\Bigl[\Bigl(\E\bigl[W_J(i)W_M(k)\,\bigl|\,\F_{J\Delta M}\bigr]\Bigr)^2\Bigr]\biggl)^{1/2}\\
&\hspace{3cm}  \biggl(\E\Bigl[\Bigl(\E\bigl[W_L(i)W_K(k)\,\bigl|\,\F_{L\Delta K}\bigr]\Bigr)^2\Bigr]\biggl)^{1/2}\\
&\leq \sum_{(J,K,L,M)\in\mathcal{T}_{i,k}}\sigma_J(i)\sigma_K(k)\sigma_L(i)\sigma_M(k)=\tau_{i,k}\,.
 \end{align*}
Thus, the claim follows.
\end{proof}

\noindent\textit{Proof of Proposition \ref{taubound}.} 
In order to prove Proposition \ref{taubound} let us review the following concepts and notation, introduced in \cite{deJo89}. For a quadruple $(J_1,J_2,J_3,J_4)\in\D_d^4$ write 
\begin{equation*}
 I:=J_1\cup J_2\cup J_3\cup J_4=\{i_1,\dotsc,i_r\}\quad\text{with}\quad 1\leq i_1<i_2<\dotsc <i_r\leq n
 \end{equation*}
and define the \textit{shadow} $(J_1',J_2',J_3',J_4')$ of $(J_1,J_2,J_3,J_4)$ by 
\begin{equation*}
 J_l':=\bigl\{a\in\{1,\dotsc,r\}\,:\, i_a\in J_l\bigr\}\,,\quad 1\leq l\leq4\,.
\end{equation*}
Note that since we have the equivalence
\begin{equation*}
 a\in J_l'\Leftrightarrow i_a\in J_l
\end{equation*}
the sets $J_l'$ satisfy obvious relations like 
\begin{equation}\label{shadint}
\abs{J_l'}=\abs{J_l}=d\,,\quad\abs{J_l'\cap J_m'}=\abs{J_l\cap J_m}\quad\text{etc.} 
\end{equation}
and that a quadruple $(J_1,J_2,J_3,J_4)\in\D_d^4$ is completely determined by its shadow and by $I=\bigcup_{l=1}^4 J_l$.
Note also that if $(J_1,J_2,J_3,J_4)\in\mathcal{T}_d$, then we have 
\begin{equation*}
 J_1'\cup J_2'\cup J_3'\cup J_4'=\{1,\dotsc,r\}
\end{equation*}
for some $r\in\{d,d+1,\dotsc,2d-1\}$ and $J_i\cap J_k\not=\emptyset$ for all $i,k=1,2,3,4$. Indeed, if, for instance, $J_1\cap J_2$ were empty and $j_0\in J_i$ for at least three values of $i\in\{1,2,3,4\}$, then necessarily
$j_0\in J_3\cap J_4$ implying $\abs{J_3\cup J_4}\leq 2d-1$. Hence, $J_1\cup J_2\not\subseteq J_3\cup J_4$ because $\abs{J_1\cup J_2}=2d$ by disjointness. Thus, $(J_1,J_2,J_3,J_4)$ has a free index and, hence, cannot be in $\mathcal{T}_d$.
By the above observation \eqref{shadint}, this immediateley implies that also $J_i'\cap J_k'\not=\emptyset$ for all $i,k=1,2,3,4$.

In general, we call a quadruple of sets $\mathbb{F}=(F_1,F_2,F_3,F_4)$ a \textit{shadow} (a \textit{$d$-shadow}) if there is an $r\in\{d,d+1,\dotsc,2d-1\}$ such that 
$F:=F_1\cup F_2\cup F_3\cup F_4=\{1,\dotsc,r\}$ and $\abs{F_l}=d$ for $l=1,2,3,4$. We call $r$ the \textit{size} of the shadow $\mathbb{F}$. We say that the shadow $\mathbb{F}$ is \textit{induced} by the quadruple $(J_1,J_2,J_3,J_4)\in\D_d^4$, if 
$\mathbb{F}=(J_1',J_2',J_3',J_4')$. We write $\mathbb{F}(J_1,J_2,J_3,J_4)$ for the shadow induced by $(J_1,J_2,J_3,J_4)$.
If $\mathbb{F}'=(F_1',F_2',F_3',F_4')$ is another $d$-shadow with $F':=F_1'\cup F_2'\cup F_3'\cup F_4'=\{1,\dotsc,r'\}$, then we say that $\mathbb{F}$ and $\mathbb{F}'$ are \textit{equivalent} and write $\mathbb{F}\sim\mathbb{F}'$, if $r=r'$ and there is a permutation $\sigma\in\mathbb{S}_r$ such that 
\begin{equation}\label{eqshad}
F_l'=\sigma(F_l)\quad\text{for }l=1,2,3,4.
\end{equation}
We denote the latter fact by $\mathbb{F}'=\mathbb{F}_\sigma$.
This clearly defines an equivalence relation on the set of $d$-shadows and we denote by $[\mathbb{F}]_{\sim}$ the equivalence class of $\mathbb{F}$. We further denote by $\gamma(\mathbb{F})$ the number of permutations 
$\sigma\in\mathbb{S}_r$ that leave $\mathbb{F}$ fixed in the sense that 
\begin{equation}\label{ershadow}
\sigma(F_l)=F_l\quad\text{for all }l=1,2,3,4.
\end{equation}
The set of these permutaions is just the \textit{stabilizer} of $\mathbb{F}$ with respect to the natural action of $\mathbb{S}_r$ on the set of $d$-shadows of size $r$. Note that, for $\mathbb{F}'\sim \mathbb{F}$, we have $\gamma(\mathbb{F})=\gamma(\mathbb{F}')$ and that 
$\gamma(\mathbb{F})$ also gives the number of permutations $\sigma$ such that \eqref{eqshad} holds.
Let us define the function $g:[n]^d\rightarrow\R$ by 
\begin{equation*}
 g(j_1,\dotsc,j_d):=\begin{cases}
                     \sigma_{\{j_1,\dotsc,j_d\}}\,,&\text{if }\babs{\{j_1,\dotsc,j_d\}}=d\\
                     0\,,&\text{otherwise.}
                    \end{cases}
\end{equation*}
Then, $g$ is a symmetric function vanishing on the complement $\Delta^c=[n]^d\setminus\Delta$ of 
\begin{equation*}
 \Delta:=\Delta_{d}^{(n)}:=\{(j_1,\dotsc,j_d)\in[n]^d\,:\,j_l\not=j_m\text{ whenever }l\not=m\}\,.
\end{equation*}
Further, for a shadow $\mathbb{F}=(F_1,F_2,F_3,F_4)$ which is induced by some quadruple $(J_1,J_2,J_3,J_4)\in\mathcal{D}_d^4$ and with $F:=F_1\cup F_2\cup F_3\cup F_4=\{1,\dotsc,r\}$ and $\pi_{F_l}$ being the natural
projection $[n]^F\rightarrow [n]^{F_l}$ given by $(j_a)_{a\in F}\mapsto(j_a)_{a\in F_l}$, define $G_{F_l}:[n]^F\rightarrow\R$ by $G_{F_l}:=g\circ\pi_{F_l}$. Here, we tacitly identify $[n]^d$ with $[n]^{F_l}$ and $[n]^r$ with $[n]^F$. 
\begin{lemma}\label{shadow}
 Let $\mathbb{F}=(F_1,F_2,F_3,F_4)$ be a $d$-shadow of size $r$ which is induced by some quadruple $(J_1,J_2,J_3,J_4)\in\mathcal{T}$. Then, we have the bound
 \begin{align*}
  \sum_{\substack{(J,K,L,M)\in\mathcal{T}:\\\mathbb{F}(J,K,L,M)\in[\mathbb{F}]_\sim}}\sigma_J\sigma_K\sigma_L\sigma_M&\leq\frac{d!(d-1)!}{\gamma(\mathbb{F})} \rho_n^2\,.
\end{align*}
 \end{lemma}
 
 \begin{proof}
 For ease of notation, in this proof we use bold letters $\mathbf{a}$ to denote tuples $\mathbf{a}=(a_1,\dotsc,a_s)\in[n]^s$, where $s$ is some natural number. Also, for two such tuples $\mathbf{a}=(a_1,\dotsc,a_s)\in[n]^s$ and 
 $\mathbf{b}=(b_1,\dotsc,b_t)\in[n]^t$ we write $\mathbf{a}\cap \mathbf{b}\not=\emptyset$ if there are indices $1\leq i\leq s$ and $1\leq j\leq t$ such that $a_i=b_j$, i.e. if 
 \begin{equation*}
  \{a_1,\dotsc,a_s\}\cap\{b_1,\dotsc,b_t\}\not=\emptyset\,.
 \end{equation*}

  We begin the proof with the remark that 
\begin{align*}
&\quad\sum_{(i_1,\dotsc,i_r)\in[n]^r}G_{F_1}(i_1,\dotsc,i_r)G_{F_2}(i_1,\dotsc,i_r)G_{F_3}(i_1,\dotsc,i_r)G_{F_4}(i_1,\dotsc,i_r)\\
&\geq\gamma(\mathbb{F})\sum_{\substack{(J,K,L,M)\in\mathcal{T}:\\\mathbb{F}(J,K,L,M)\in[\mathbb{F}]_\sim}}\sigma_J\sigma_K\sigma_L\sigma_M\,.
\end{align*}
This follows from 
\begin{align*}
 \sum_{(i_1,\dotsc,i_r)\in[n]^r}\prod_{l=1}^4G_{F_l}(i_1,\dotsc,i_r)
 &\geq\sum_{(i_1,\dotsc,i_r)\in[n]^r_{\not=}}\prod_{l=1}^4G_{F_l}(i_1,\dotsc,i_r)\\
 &=\sum_{1\leq j_1<\dotsc<j_r\leq n}\sum_{\sigma\in\mathbb{S}_r}\prod_{l=1}^4g\bigl(\pi_{F_l}(j_{\sigma(1)},\dotsc,j_{\sigma(r)})\bigr)\\
 &=\sum_{1\leq j_1<\dotsc<j_r\leq n}\sum_{\sigma\in\mathbb{S}_r}\prod_{l=1}^4g\bigl(\pi_{\sigma(F_l)}(j_1,\dotsc,j_r)\bigr)\\
 &=\gamma(\mathbb{F})\sum_{\mathbb{F}'\in[\mathbb{F}]_\sim}\sum_{1\leq j_1<\dotsc<j_r\leq n}\prod_{l=1}^4G_{F_l'}(j_1,\dotsc,j_r)\\
 &=\gamma(\mathbb{F})\sum_{\mathbb{F}'\in[\mathbb{F}]_\sim}\sum_{\substack{(J,K,L,M)\in\mathcal{T}:\\\mathbb{F}(J,K,L,M)=\mathbb{F}'}}\sigma_J\sigma_K\sigma_L\sigma_M\\
 &=\gamma(\mathbb{F})\sum_{\substack{(J,K,L,M)\in\mathcal{T}:\\\mathbb{F}(J,K,L,M)\in[\mathbb{F}]_\sim}}\sigma_J\sigma_K\sigma_L\sigma_M\,.
\end{align*}
Here, we used the notation $[n]^r_{\not=}$ for the set of all tuples $(i_1,\dotsc,i_r)\in[n]^r$ such that $i_j\not=i_k$ whenever $j\not=k$.
Hence, it suffices to show that we always have the bound
\begin{align*}
 \sum_{(i_1,\dotsc,i_r)\in[n]^r}\prod_{l=1}^4G_{F_l}(i_1,\dotsc,i_r)\leq d!(d-1)! \rho_n^2
\end{align*}
if $\mathbb{F}$ is as in the statement of the Lemma.\\

 We first treat the simple cases that either two or all of the sets $F_l$, $l=1,2,3,4$, are equal. Note that the case of exactly three equal sets is vacuous for a quadruple in $\mathcal{T}$.
  Assume first that e.g. $F_3\not=F_1=F_2\not=F_4$. It might be that also $F_3=F_4$ but this is immaterial. Then, we have 
  \begin{align*}
   &\quad\sum_{(i_1,\dotsc,i_r)\in[n]^r}G_{F_1}(i_1,\dotsc,i_r)G_{F_2}(i_1,\dotsc,i_r)G_{F_3}(i_1,\dotsc,i_r)G_{F_4}(i_1,\dotsc,i_r)\\
   &=\sum_{(i_1,\dotsc,i_r)\in[n]^F}G_{F_1}^2(i_1,\dotsc,i_r)G_{F_3}(i_1,\dotsc,i_r)G_{F_4}(i_1,\dotsc,i_r)\\
  &=\sum_{\mathbf{j}\in[n]^{F_1}}g(\mathbf{j})^2\sum_{\mathbf{k}\in[n]^{F\setminus F_1}}G_{F_3}(\mathbf{j},\mathbf{k})G_{F_4}(\mathbf{j},\mathbf{k})\\
 &\leq \sum_{\mathbf{j}\in[n]^{F_1}}g(\mathbf{j})^2 \biggl(\sum_{\mathbf{k}\in[n]^{F\setminus F_1}}G_{F_3}^2(\mathbf{j},\mathbf{k})\biggr)^{1/2}
 \biggl(\sum_{\mathbf{k}\in[n]^{F\setminus F_1}}G_{F_4}^2(\mathbf{j},\mathbf{k})\biggr)^{1/2}\\
 &=\sum_{\mathbf{j}\in[n]^{F_1}}g(\mathbf{j})^2\biggl(\sum_{\substack{\mathbf{l}\in[n]^{F_3}: \mathbf{l}\cap\mathbf{j}\not=\emptyset}}g(\mathbf{l})^2\biggr)^{1/2}
 \biggl(\sum_{\substack{\mathbf{m}\in[n]^{F_4}: \mathbf{m}\cap\mathbf{j}\not=\emptyset}}g(\mathbf{m})^2\biggr)^{1/2}\\
 &\leq (d-1)!\rho_n^2 \sum_{\mathbf{j}\in[n]^{F_1}}g(\mathbf{j})^2\\
 &=d!(d-1)!\rho_n^2 
  \end{align*}
Note that the second inequality follows from the fact that $F_1\cap F_3\not=\emptyset$ and $F_1\cap F_4\not=\emptyset$ in this case as well as by the definition of $\rho_n^2$.
 If $F_1=F_2=F_3=F_4$, then we have $r=d$ and
 \begin{align*}
  &\quad\sum_{(i_1,\dotsc,i_r)\in[n]^r}G_{F_1}(i_1,\dotsc,i_r)G_{F_2}(i_1,\dotsc,i_r)G_{F_3}(i_1,\dotsc,i_r)G_{F_4}(i_1,\dotsc,i_r)\\
  &=\quad\sum_{(j_1,\dotsc,j_d)\in[n]^d}g(j_1,\dotsc,j_d)^4\\
  &\leq \sum_{j_1=1}^n \max_{(j_2,\dotsc,j_d)\in[n]^{d-1}}g(j_1,\dotsc,j_d)^2\sum_{(k_2,\dotsc,k_d)\in[n]^{d-1}}g(j_1,k_2,\dotsc,k_d)^2\\
  &\leq (d-1)!\rho_n^2\sum_{j_1=1}^n \max_{(j_2,\dotsc,j_d)\in[n]^{d-1}}g(j_1,\dotsc,j_d)^2\\
  &\leq (d-1)!\rho_n^2 \sum_{(j_1,\dotsc,j_d)\in[n]^d}g(j_1,\dotsc,j_d)^2\\
  &= d!(d-1)!\rho_n^2\,.
 \end{align*}
For the remainder of this proof we may thus assume that the sets $F_l$, $l=1,2,3,4$, are pairwise different. Then, using the Cauchy-Schwarz inequality, we can bound 
\begin{align}\label{sh1}
 &\quad\sum_{(i_1,\dotsc,i_r)\in[n]^r}G_{F_1}(i_1,\dotsc,i_r)G_{F_2}(i_1,\dotsc,i_r)G_{F_3}(i_1,\dotsc,i_r)G_{F_4}(i_1,\dotsc,i_r)\notag\\
 &=\sum_{\mathbf{j}\in[n]^{F_1}}g(\mathbf{j})\sum_{\mathbf{k}\in[n]^{F\setminus F_1}}G_{F_2}(\mathbf{j},\mathbf{k})G_{F_3}(\mathbf{j},\mathbf{k})G_{F_4}(\mathbf{j},\mathbf{k})\notag\\
 &\leq \biggl(\sum_{\mathbf{j}\in[n]^{F_1}}g(\mathbf{j})^2\biggr)^{1/2} 
 \biggl(\sum_{\mathbf{j}\in[n]^{F_1}}\Bigl(\sum_{\mathbf{k}\in[n]^{F\setminus F_1}}G_{F_2}(\mathbf{j},\mathbf{k})G_{F_3}(\mathbf{j},\mathbf{k})G_{F_4}(\mathbf{j},\mathbf{k})\Bigr)^2\biggr)^{1/2}\notag\\
 &=\sqrt{d!}\biggl(\sum_{\mathbf{j}\in[n]^{F_1}}\Bigl(\sum_{\mathbf{k}\in[n]^{F\setminus F_1}}G_{F_2}(\mathbf{j},\mathbf{k})G_{F_3}(\mathbf{j},\mathbf{k})G_{F_4}(\mathbf{j},\mathbf{k})\Bigr)^2\biggr)^{1/2}
\end{align}
Thus, it remains to bound the quantity 
\begin{equation*}
 A:=\sum_{\mathbf{j}\in[n]^{F_1}}\Bigl(\sum_{\mathbf{k}\in[n]^{F\setminus F_1}}G_{F_2}(\mathbf{j},\mathbf{k})G_{F_3}(\mathbf{j},\mathbf{k})G_{F_4}(\mathbf{j},\mathbf{k})\Bigr)^2\,.
\end{equation*}
Let us distinguish the following cases. 
\begin{enumerate}[1)]
 \item Each element in $F=F_1\cup F_2\cup F_3\cup F_4$ appears in at least three of the sets $F_1,F_2,F_3,F_4$. This implies that 
 \begin{equation*}
  F\setminus F_k= F_j\setminus F_k\quad\text{for all distinct }j,k\in\{2,3,4\}\,.
 \end{equation*}
Then, using Cauchy-Schwarz, we can bound
\begin{align*}
 A&\leq \sum_{\mathbf{j}\in[n]^{F_1}}\biggl(\sum_{\mathbf{k}\in[n]^{F\setminus F_1}}G_{F_2}^2(\mathbf{j},\mathbf{k})
 \sum_{\mathbf{k}\in[n]^{F\setminus F_1}}G_{F_3}^2(\mathbf{j},\mathbf{k})G_{F_4}^2(\mathbf{j},\mathbf{k})\biggr)\\
 &= \sum_{\mathbf{j}\in[n]^{F_1}}\biggl(\sum_{\substack{\mathbf{l}\in[n]^{F_2}: \mathbf{l}\cap\mathbf{j}\not=\emptyset}}g(\mathbf{l})^2
 \sum_{\mathbf{k}\in[n]^{F\setminus F_1}}G_{F_3}^2(\mathbf{j},\mathbf{k})G_{F_4}^2(\mathbf{j},\mathbf{k})\biggr)\\
 &\leq (d-1)!\rho_n^2\sum_{\mathbf{j}\in[n]^{F_1}}\sum_{\mathbf{k}\in[n]^{F\setminus F_1}}G_{F_3}^2(\mathbf{j},\mathbf{k})G_{F_4}^2(\mathbf{j},\mathbf{k})\\
 &=(d-1)!\rho_n^2\sum_{\mathbf{k}\in[n]^{F\setminus F_1}}\sum_{\mathbf{l}\in[n]^{F_1\cap F_3\cap F_4}}\sum_{\mathbf{a}\in[n]^{(F_3\cap F_1)\setminus F_4}}g(\mathbf{k},\mathbf{l},\mathbf{a})^2
 \sum_{\mathbf{b}\in[n]^{(F_4\cap F_1)\setminus F_3}}g(\mathbf{k},\mathbf{l},\mathbf{b})^2\\
 &\leq \bigl((d-1)!\bigr)^2 \rho_n^4\sum_{\mathbf{k}\in[n]^{F\setminus F_1}}\sum_{\mathbf{l}\in[n]^{F_1\cap F_3\cap F_4}}\sum_{\mathbf{a}\in[n]^{(F_3\cap F_1)\setminus F_4}}g(\mathbf{k},\mathbf{l},\mathbf{a})^2\\
 &= \bigl((d-1)!\bigr)^2 \rho_n^4\sum_{\mathbf{j}\in[n]^{F_3}}g(\mathbf{j})^2\\
 &=d!\bigl((d-1)!\bigr)^2 \rho_n^4\,.
\end{align*}
Note that we have used the fact that 
\begin{equation*}
 (F_4\cap F_1)\setminus F_3=F\setminus F_3\not=\emptyset
\end{equation*}
to obtain the last inequality. 
\item There is an element $j_0\in F=F_1\cup F_2\cup F_3\cup F_4$ which is contained in exactly two of the sets $F_1,F_2,F_3,F_4$. We may assume that $j_0\in F_1$. We claim that then there are distinct indices 
$j,k\in\{2,3,4\}$ such that 
\begin{equation*}
 F_1\not\subseteq F_j\cup F_k\,.
\end{equation*}
Indeed, we have 
\begin{equation*}
 (F_2\cup F_3)\cap(F_2\cup F_4)\cap(F_3\cup F_4)=(F_2\cap F_3)\cup(F_2\cap F_4)\cup(F_3\cap F_4)
\end{equation*}
and, hence, $j_0$ cannot be contained in the set on the right hand side. Thus, let us assume that $F_1\not\subseteq F_3\cup F_4$. We obtain that 
\begin{align*}
A&\leq\sum_{\mathbf{j}\in[n]^{F_1\setminus F_2}}\sum_{\mathbf{a}\in[n]^{F_1\cap F_2}}
\biggl(\sum_{\mathbf{k}\in[n]^{F_2\setminus F_1}}g(\mathbf{a},\mathbf{k})
\sum_{\mathbf{l}\in[n]^{(F_3\cup F_4)\setminus(F_1\cup F_2)}}G_{F_3}(\mathbf{j},\mathbf{a},\mathbf{k},\mathbf{l})
G_{F_4}(\mathbf{j},\mathbf{a},\mathbf{k},\mathbf{l})\biggr)^2\\
&\leq \sum_{\mathbf{j}\in[n]^{F_1\setminus F_2}}\sum_{\mathbf{a}\in[n]^{F_1\cap F_2}}\Bigl(\sum_{\mathbf{k}\in[n]^{F_2\setminus F_1}}g(\mathbf{a},\mathbf{k})^2\Bigr)\\
&\quad\cdot\sum_{\mathbf{k}\in[n]^{F_2\setminus F_1}}\biggl(\sum_{\mathbf{l}\in[n]^{(F_3\cup F_4)\setminus(F_1\cup F_2)}}G_{F_3}(\mathbf{j},\mathbf{a},\mathbf{k},\mathbf{l})
G_{F_4}(\mathbf{j},\mathbf{a},\mathbf{k},\mathbf{l})\biggr)^2\\
&=\sum_{\mathbf{j}\in[n]^{F_1\setminus F_2}}\sum_{\mathbf{a_1}\in[n]^{F_1\cap F_2\cap(F_3\cup F_4)}}
\biggl(\sum_{\mathbf{a_2}\in[n]^{F_1\cap F_2\setminus(F_3\cup F_4)}}\sum_{\mathbf{k}\in[n]^{F_2\setminus F_1}}
g(\mathbf{a_1},\mathbf{a_2},\mathbf{k})^2\biggr)\\
&\hspace{3cm}\cdot\sum_{\mathbf{k}\in[n]^{F_2\setminus F_1}}\biggl(\sum_{\mathbf{l}\in[n]^{(F_3\cup F_4)\setminus(F_1\cup F_2)}}G_{F_3}(\mathbf{j},\mathbf{a_1},\mathbf{a_2^*},\mathbf{k},\mathbf{l})
G_{F_4}(\mathbf{j},\mathbf{a_1},\mathbf{a_2^*},\mathbf{k},\mathbf{l})\biggr)^2\\
&\leq (d-1)!\rho_n^2\sum_{\mathbf{j}\in[n]^{F_1\setminus F_2}}\sum_{\mathbf{a_1}\in[n]^{F_1\cap F_2\cap(F_3\cup F_4)}}
\sum_{\mathbf{k}\in[n]^{F_2\setminus F_1}}\\
&\hspace{3cm}\cdot\biggl(\sum_{\mathbf{l}\in[n]^{(F_3\cup F_4)\setminus(F_1\cup F_2)}}G_{F_3}(\mathbf{j},\mathbf{a_1},\mathbf{a_2^*},\mathbf{k},\mathbf{l})
G_{F_4}(\mathbf{j},\mathbf{a_1},\mathbf{a_2^*},\mathbf{k},\mathbf{l})\biggr)^2\,,
\end{align*}
where 
\[\mathbf{a_2^*} \in[n]^{F_1\cap F_2\setminus(F_3\cup F_4)}\]
is arbitrary but fixed.
Now note that due to the fact that $\mathbb{F}$ is induced by some quadruple in $\mathcal{T}$ we have
\begin{align*}
F_3\cup F_4&=(F_1\setminus F_2)\cup (F_2\setminus F_1)\cup\bigl[(F_3\cup F_4)\cap F_1\cap F_2\bigr]
\cup\bigl[(F_3\cup F_4)\setminus( F_1\cup F_2)\bigr]\\
&=(F_1\setminus F_2)\cup (F_2\setminus F_1)\cup\bigl[(F_3\cup F_4)\cap F_1\cap F_2\bigr]
\cup\bigl[(F_3\cap F_4)\setminus( F_1\cup F_2)\bigr]\,,
\end{align*}
where the union on the right hand side is disjoint. Thus, the last bound becomes 
\begin{align*}
A&\leq(d-1)!\rho_n^2\sum_{\mathbf{m}\in[n]^{(F_3\cup F_4)\cap(F_1\cup F_2)}}
\sum_{\mathbf{l}\in[n]^{(F_3\cap F_4)\setminus(F_1\cup F_2)}}G_{F_3}^2(\mathbf{m},\mathbf{l})\\
&\hspace{3cm}\cdot\sum_{\mathbf{p}\in[n]^{(F_3\cap F_4)\setminus(F_1\cup F_2)}}G_{F_4}^2(\mathbf{m},\mathbf{p})\\
&=(d-1)!\rho_n^2\sum_{\mathbf{a}\in[n]^{(F_3\cap F_4)\cap(F_1\cup F_2)}}
\sum_{\mathbf{b}\in[n]^{(F_3\setminus F_4)\cap(F_1\cup F_2)}}\sum_{\mathbf{l}\in[n]^{(F_3\cap F_4)\setminus(F_1\cup F_2)}}g(\mathbf{a},\mathbf{b},\mathbf{l})^2\\
&\hspace{3cm}\sum_{\mathbf{c}\in[n]^{(F_4\setminus F_3)\cap(F_1\cup F_2)}}\sum_{\mathbf{p}\in[n]^{(F_3\cap F_4)\setminus(F_1\cup F_2)}}g(\mathbf{a},\mathbf{c},\mathbf{p})^2\\
&\leq \bigl((d-1)!\bigr)^2\rho_n^4\sum_{\mathbf{j}\in[n]^{F_3}}g(\mathbf{j})^2\\
&=d!\bigl((d-1)!\bigr)^2\rho_n^4
\end{align*}
\end{enumerate}
\end{proof}

\begin{proof}[End of the proof of Proposition \eqref{taubound}]
 Let $\mathbb{F}_1,\dotsc,\mathbb{F}_s$ be a complete system of pairwise non-equivalent $d$-shadows which are induced by quadruples $(J,K,L,M)\in\mathcal{T}$. 
 Then, clearly, $s$ is independent of $n$ and by Lemma \ref{shadow} we have 
 \begin{align*}
  \tau&=\sum_{(J,K,L,M)\in\mathcal{T}}\sigma_J\sigma_K\sigma_L\sigma_M=\sum_{j=1}^s  \sum_{\substack{(J,K,L,M)\in\mathcal{T}:\\\mathbb{F}(J,K,L,M)\in[\mathbb{F}_j]_\sim}}\sigma_J\sigma_K\sigma_L\sigma_M\\
  &\leq \Bigl(d!(d-1)!\sum_{j=1}^s\gamma(\mathbb{F}_j)^{-1}\Bigr)\rho_n^2
 \end{align*}
so that we can let 
 \begin{equation}\label{Cd}
  C_d:=d!(d-1)!\sum_{j=1}^s\gamma(\mathbb{F}_j)^{-1}
 \end{equation}
which is independent of $n$.\\
\end{proof}

\begin{remark}\label{taurem}
Using the fact that the equivalence class of a shadow \\$\mathbb{F}=(F_1,F_2,F_3,F_4)$ is determined by the cardinalities of all finite intersections of the sets $F_1,F_2,F_3,F_4$, one 
can get an upper bound on the number $s$ of all equivalence classes of shadows induced  by quadruples in $\mathcal{T}$. Using that $\gamma(\mathbb{F})\geq1$ immediately gives a crude bound on $C_d$.
It is not difficult to verify that $C_2=13$ by distinguishing all possible cases. Furthermore, by some clever combinatorial argument, it might be possible to compute sharp bounds on $C_d$ starting from \eqref{Cd}. This would be of great interest for deriving limit theorems in situations where $d=d_n\to\infty$ with $n$. We leave this as an interesting problem for possible future work.
\end{remark}

\begin{proof}[Idea of the proof of Proposition \ref{gentaubound}]
The proof of Proposition \ref{taubound} can be easily generalized to the present situation by introducing the concept of a \textit{$(p_i,p_k)$-shadow} corresponding to a quadruple $(J_1,J_2,J_3,J_4)\in\D_{i,k}^4$ 
and following exactly the same lines of the proof. We have, however, refrained from giving the proof in this more general situation for mainly two reasons. Firstly, the proof of Proposition \ref{taubound} already involves 
a lot of notation and introducing even more of it might make the argument less transparent. Secondly, and more importantly, the precise dependence of the constant $C_{i,k}$ on $p_i$ and $p_k$ would be more complicated and less explicit than 
the formula given by \eqref{Cd} which can be exactly evaluated for small values of $d$ and, as mentioned in Remark \ref{taurem}, might be suitably bounded for general $d$. \\
\end{proof}

\normalem
\bibliography{dejong}{}

\begin{thebibliography}{CNPP16}

\bibitem[ACP14]{ACP}
E.~Azmoodeh, S.~Campese, and G.~Poly.
\newblock Fourth {M}oment {T}heorems for {M}arkov diffusion generators.
\newblock {\em J. Funct. Anal.}, 266(4):2341--2359, 2014.

\bibitem[Ari13]{Ariz}
O.~Arizmendi.
\newblock Convergence of the fourth moment and infinite divisibility.
\newblock {\em Probab. Math. Statist.}, 33(2):201--212, 2013.

\bibitem[Bol82]{Bolt82}
E.~Bolthausen.
\newblock Exact convergence rates in some martingale central limit theorems.
\newblock {\em Ann. Probab.}, 10(3):672--688, 1982.

\bibitem[BP14a]{BP-geo}
S.~Bourguin and G.~Peccati.
\newblock Portmanteau inequalities on the {P}oisson space: mixed regimes and
  multidimensional clustering.
\newblock {\em Electron. J. Probab.}, 19:no. 66, 42, 2014.

\bibitem[BP14b]{BP-free}
S.~Bourguin and G.~Peccati.
\newblock Semicircular limits on the free {P}oisson chaos: counterexamples to a
  transfer principle.
\newblock {\em J. Funct. Anal.}, 267(4):963--997, 2014.

\bibitem[CFR11]{CFR11}
S.~Chatterjee, J.~Fulman, and A.~R{\"o}llin.
\newblock Exponential approximation by {S}tein's method and spectral graph
  theory.
\newblock {\em ALEA Lat. Am. J. Probab. Math. Stat.}, 8:197--223, 2011.

\bibitem[CGS11]{CGS}
L.~H.~Y. Chen, L.~Goldstein, and Q.-M. Shao.
\newblock {\em Normal approximation by {S}tein's method}.
\newblock Probability and its Applications (New York). Springer, Heidelberg,
  2011.

\bibitem[CM08]{ChaMe08}
S.~Chatterjee and E.~Meckes.
\newblock Multivariate normal approximation using exchangeable pairs.
\newblock {\em ALEA Lat. Am. J. Probab. Math. Stat.}, 4:257--283, 2008.

\bibitem[CNPP16]{CNPP}
S.~Campese, I.~Nourdin, G.~Peccati, and G.~Poly.
\newblock Multivariate {G}aussian approximations on {M}arkov chaoses.
\newblock {\em Electron. Commun. Probab.}, 21:Paper No. 48, 9, 2016.

\bibitem[CS11]{ChSh}
S.~Chatterjee and Q.-M. Shao.
\newblock Nonnormal approximation by {S}tein's method of exchangeable pairs
  with application to the {C}urie-{W}eiss model.
\newblock {\em Ann. Appl. Probab.}, 21(2):464--483, 2011.

\bibitem[dJ89]{deJo89}
P.~de~Jong.
\newblock {\em Central limit theorems for generalized multilinear forms},
  volume~61 of {\em CWI Tract}.
\newblock Stichting Mathematisch Centrum, Centrum voor Wiskunde en Informatica,
  Amsterdam, 1989.

\bibitem[dJ90]{deJo90}
P.~de~Jong.
\newblock A central limit theorem for generalized multilinear forms.
\newblock {\em J. Multivariate Anal.}, 34(2):275--289, 1990.

\bibitem[DM83]{DynMan83}
E.~B. Dynkin and A.~Mandelbaum.
\newblock Symmetric statistics, {P}oisson point processes, and multiple
  {W}iener integrals.
\newblock {\em Ann. Statist.}, 11(3):739--745, 1983.

\bibitem[D{\"o}b12]{Doe12c}
C.~D{\"o}bler.
\newblock {New developments in Stein's method with applications}.
\newblock 2012.
\newblock (Ph.D.)-Thesis Ruhr-Universit\"at Bochum.

\bibitem[D{\"o}b15]{DoeBeta}
C.~D{\"o}bler.
\newblock {Stein's method of exchangeable pairs for the Beta distribution and
  generalizations}.
\newblock {\em Electron. J. Probab.}, 20:no. 109, 1--34, 2015.

\bibitem[EDP08]{ElDP}
O.~El-Dakkak and G.~Peccati.
\newblock Hoeffding decompositions and urn sequences.
\newblock {\em Ann. Probab.}, 36(6):2280--2310, 2008.

\bibitem[EDPP14]{ElDPP}
O.~El-Dakkak, G.~Peccati, and I.~Pr{{\"u}}nster.
\newblock Exchangeable {H}oeffding decompositions over finite sets: a
  combinatorial characterization and counterexamples.
\newblock {\em J. Multivariate Anal.}, 131:51--64, 2014.

\bibitem[EL10]{EiLo10}
P.~Eichelsbacher and M.~L{\"o}we.
\newblock Stein's method for dependent random variables occurring in
  statistical mechanics.
\newblock {\em Electron. J. Probab.}, 15:no. 30, 962--988, 2010.

\bibitem[ET14]{EiThae14}
P.~Eichelsbacher and C.~Th{\"a}le.
\newblock New {B}erry-{E}sseen bounds for non-linear functionals of {P}oisson
  random measures.
\newblock {\em Electron. J. Probab.}, 19:no. 102, 25, 2014.

\bibitem[FR13]{FulRos13}
J.~Fulman and N.~Ross.
\newblock Exponential approximation and {S}tein's method of exchangeable pairs.
\newblock {\em ALEA Lat. Am. J. Probab. Math. Stat.}, 10(1):1--13, 2013.

\bibitem[FT16]{FT}
T.~Fissler and C.~Th{\"a}le.
\newblock A four moments theorem for gamma limits on a {P}oisson chaos.
\newblock {\em ALEA Lat. Am. J. Probab. Math. Stat.}, 13(1):163--192, 2016.

\bibitem[Gre77]{Greg77}
G.~G. Gregory.
\newblock Large sample theory for {$U$}-statistics and tests of fit.
\newblock {\em Ann. Statist.}, 5(1):110--123, 1977.

\bibitem[Hae88]{Haeu88}
E.~Haeusler.
\newblock On the rate of convergence in the central limit theorem for
  martingales with discrete and continuous time.
\newblock {\em Ann. Probab.}, 16(1):275--299, 1988.

\bibitem[HB70]{HeyBr70}
C.~C. Heyde and B.~M. Brown.
\newblock On the departure from normality of a certain class of martingales.
\newblock {\em Ann. Math. Statist.}, 41:2161--2165, 1970.

\bibitem[Hoe48]{Hoeffding}
W.~Hoeffding.
\newblock A class of statistics with asymptotically normal distribution.
\newblock {\em Ann. Math. Statistics}, 19:293--325, 1948.

\bibitem[Jan97]{J-book}
S.~Janson.
\newblock {\em Gaussian {H}ilbert spaces}, volume 129 of {\em Cambridge Tracts
  in Mathematics}.
\newblock Cambridge University Press, Cambridge, 1997.

\bibitem[JJ86]{JJ}
S.~R. Jammalamadaka and S.~Janson.
\newblock Limit theorems for a triangular scheme of {$U$}-statistics with
  applications to inter-point distances.
\newblock {\em Ann. Probab.}, 14(4):1347--1358, 1986.

\bibitem[Kal02]{Kal}
O.~Kallenberg.
\newblock {\em Foundations of modern probability}.
\newblock Probability and its Applications (New York). Springer-Verlag, New
  York, second edition, 2002.

\bibitem[KB94]{KB-book}
V.~S. Koroljuk and Yu.~V. Borovskich.
\newblock {\em Theory of {$U$}-statistics}, volume 273 of {\em Mathematics and
  its Applications}.
\newblock Kluwer Academic Publishers Group, Dordrecht, 1994.
\newblock Translated from the 1989 Russian original by P. V. Malyshev and D. V.
  Malyshev and revised by the authors.

\bibitem[KNS12]{KNPS}
T.~Kemp, G.~Nourdin, I.and~Peccati, and R.~Speicher.
\newblock Wigner chaos and the fourth moment.
\newblock {\em Ann. Probab.}, 40(4):1577--1635, 2012.

\bibitem[KR82]{KR}
S.~Karlin and Y.~Rinott.
\newblock Applications of {ANOVA} type decompositions for comparisons of
  conditional variance statistics including jackknife estimates.
\newblock {\em Ann. Statist.}, 10(2):485--501, 1982.

\bibitem[KRT]{KRT}
K.~Krokowski, A.~Reichenbachs, and C.~Th{\"a}le.
\newblock {Discrete Malliavin-Stein method: Berry-Esseen bounds for random
  graphs and percolation}.
\newblock {\em to appear in: Ann. Probab.}

\bibitem[Led12]{ledoux-aop}
M.~Ledoux.
\newblock Chaos of a {M}arkov operator and the fourth moment condition.
\newblock {\em Ann. Probab.}, 40(6):2439--2459, 2012.

\bibitem[LRP]{LRP}
R.~Lachi{\`e}ze-Rey and G.~Peccati.
\newblock {New Kolmogorov bounds for functionals of binomial point processes}.
\newblock {\em to appear in: Ann. Appl. Probab.}

\bibitem[LRP13a]{LRP1}
R.~Lachi{{\`e}}ze-Rey and G.~Peccati.
\newblock Fine {G}aussian fluctuations on the {P}oisson space, {I}:
  contractions, cumulants and geometric random graphs.
\newblock {\em Electron. J. Probab.}, 18:no. 32, 32, 2013.

\bibitem[LRP13b]{LRP2}
R.~Lachi{{\`e}}ze-Rey and G.~Peccati.
\newblock Fine {G}aussian fluctuations on the {P}oisson space {II}: rescaled
  kernels, marked processes and geometric {$U$}-statistics.
\newblock {\em Stochastic Process. Appl.}, 123(12):4186--4218, 2013.

\bibitem[Mec09]{Meck09}
E.~Meckes.
\newblock On {S}tein's method for multivariate normal approximation.
\newblock In {\em High dimensional probability {V}: the {L}uminy volume},
  volume~5 of {\em Inst. Math. Stat. Collect.}, pages 153--178. Inst. Math.
  Statist., Beachwood, OH, 2009.

\bibitem[MOO10]{MOO}
E.~Mossel, R.~O'Donnell, and K.~Oleszkiewicz.
\newblock Noise stability of functions with low influences: invariance and
  optimality.
\newblock {\em Ann. of Math. (2)}, 171(1):295--341, 2010.

\bibitem[MS75]{GinSib}
W.~G. McGinley and R.~Sibson.
\newblock Dissociated random variables.
\newblock {\em Math. Proc. Cambridge Philos. Soc.}, 77:185--188, 1975.

\bibitem[NP05]{NuaPec05}
D.~Nualart and G.~Peccati.
\newblock Central limit theorems for sequences of multiple stochastic
  integrals.
\newblock {\em Ann. Probab.}, 33(1):177--193, 2005.

\bibitem[NP09]{NouPec09a}
I.~Nourdin and G.~Peccati.
\newblock Stein's method on {W}iener chaos.
\newblock {\em Probab. Theory Related Fields}, 145(1-2):75--118, 2009.

\bibitem[NP12]{NouPecbook}
I.~Nourdin and G.~Peccati.
\newblock {\em Normal approximations with {M}alliavin calculus}, volume 192 of
  {\em Cambridge Tracts in Mathematics}.
\newblock Cambridge University Press, Cambridge, 2012.
\newblock From Stein's method to universality.

\bibitem[NPPS16]{NPPSi}
I.~Nourdin, G.~Peccati, G.~Poly, and R.~Simone.
\newblock Multidimensional limit theorems for homogeneous sums: a survey and a
  general transfer principle.
\newblock {\em ESAIM Probab. Stat.}, 20:293--308, 2016.

\bibitem[NPR10a]{NPR-aop}
I.~Nourdin, G.~Peccati, and G.~Reinert.
\newblock Invariance principles for homogeneous sums: universality of
  {G}aussian {W}iener chaos.
\newblock {\em Ann. Probab.}, 38(5):1947--1985, 2010.

\bibitem[NPR10b]{NPR-ejp}
I.~Nourdin, G.~Peccati, and G.~Reinert.
\newblock Stein's method and stochastic analysis of {R}ademacher functionals.
\newblock {\em Electron. J. Probab.}, 15:no. 55, 1703--1742, 2010.

\bibitem[NPS13]{NPSP}
I.~Nourdin, G.~Peccati, and R.~Speicher.
\newblock Multi-dimensional semicircular limits on the free {W}igner chaos.
\newblock In {\em Seminar on {S}tochastic {A}nalysis, {R}andom {F}ields and
  {A}pplications {VII}}, volume~67 of {\em Progr. Probab.}, pages 211--221.
  Birkh{\"a}user/Springer, Basel, 2013.

\bibitem[Pec04]{P-aop}
G.~Peccati.
\newblock Hoeffding-{ANOVA} decompositions for symmetric statistics of
  exchangeable observations.
\newblock {\em Ann. Probab.}, 32(3A):1796--1829, 2004.

\bibitem[Pen03]{Pen-book}
M.~Penrose.
\newblock {\em Random geometric graphs}, volume~5 of {\em Oxford Studies in
  Probability}.
\newblock Oxford University Press, Oxford, 2003.

\bibitem[PR16]{RePe-book}
G.~Peccati and M.~Reitzner, editors.
\newblock {\em Stochastic Analysis for Poisson Point Processes}.
\newblock Springer-Verlag, 2016.

\bibitem[PSTU10]{PSTU}
G.~Peccati, J.~L. Sol{{\'e}}, M.~S. Taqqu, and F.~Utzet.
\newblock {Stein's method and normal approximation of Poisson functionals}.
\newblock {\em Ann. Probab.}, 38(2):443--478, 2010.

\bibitem[PT05]{PeTu}
G.~Peccati and C.~A. Tudor.
\newblock Gaussian limits for vector-valued multiple stochastic integrals.
\newblock In {\em S{\'e}minaire de {P}robabilit{\'e}s {XXXVIII}}, volume 1857
  of {\em Lecture Notes in Math.}, pages 247--262. Springer, Berlin, 2005.

\bibitem[PT13]{PTh}
G.~Peccati and C.~Th{{\"a}}le.
\newblock Gamma limits and {$U$}-statistics on the {P}oisson space.
\newblock {\em ALEA Lat. Am. J. Probab. Math. Stat.}, 10(1):525--560, 2013.

\bibitem[PT15]{PrTo}
N.~Privault and G.L. Torrisi.
\newblock The {S}tein and {C}hen-{S}tein methods for functionals of
  non-symmetric {B}ernoulli processes.
\newblock {\em ALEA Lat. Am. J. Probab. Math. Stat.}, 12(1):309--356, 2015.

\bibitem[PZ10]{PZ-ejp}
G.~Peccati and C.~Zheng.
\newblock Multi-dimensional {G}aussian fluctuations on the {P}oisson space.
\newblock {\em Electron. J. Probab.}, 15:no. 48, 1487--1527, 2010.

\bibitem[PZ14]{PZ-ber}
G.~Peccati and C.~Zheng.
\newblock Universal {G}aussian fluctuations on the discrete {P}oisson chaos.
\newblock {\em Bernoulli}, 20(2):697--715, 2014.

\bibitem[R{\"o}l08]{Ro08}
A.~R{\"o}llin.
\newblock A note on the exchangeability condition in {S}tein's method.
\newblock {\em Statist. Probab. Lett.}, 78(13):1800--1806, 2008.

\bibitem[RR97]{RiRo97}
Y.~Rinott and V.~Rotar.
\newblock On coupling constructions and rates in the {CLT} for dependent
  summands with applications to the antivoter model and weighted
  {$U$}-statistics.
\newblock {\em Ann. Appl. Probab.}, 7(4):1080--1105, 1997.

\bibitem[RR09]{ReiRoe09}
G.~Reinert and A.~R{\"o}llin.
\newblock Multivariate normal approximation with {S}tein's method of
  exchangeable pairs under a general linearity condition.
\newblock {\em Ann. Probab.}, 37(6):2150--2173, 2009.

\bibitem[RS13]{ReSch}
M.~Reitzner and M.~Schulte.
\newblock Central limit theorems for {$U$}-statistics of {P}oisson point
  processes.
\newblock {\em Ann. Probab.}, 41(6):3879--3909, 2013.

\bibitem[RV80]{RubVi80}
H.~Rubin and R.~A. Vitale.
\newblock Asymptotic distribution of symmetric statistics.
\newblock {\em Ann. Statist.}, 8(1):165--170, 1980.

\bibitem[Sch16]{Sch-Poisson}
M.~Schulte.
\newblock Normal {A}pproximation of {P}oisson {F}unctionals in {K}olmogorov
  {D}istance.
\newblock {\em J. Theoret. Probab.}, 29(1):96--117, 2016.

\bibitem[Ser80]{ser-book}
R.~J. Serfling.
\newblock {\em Approximation theorems of mathematical statistics}.
\newblock John Wiley \& Sons, Inc., New York, 1980.
\newblock Wiley Series in Probability and Mathematical Statistics.

\bibitem[Ste86]{St86}
C.~Stein.
\newblock {\em Approximate computation of expectations}.
\newblock Institute of Mathematical Statistics Lecture Notes---Monograph
  Series, 7. Institute of Mathematical Statistics, Hayward, CA, 1986.

\bibitem[Vit92]{vitale}
R.~A. Vitale.
\newblock Covariances of symmetric statistics.
\newblock {\em J. Multivariate Anal.}, 41(1):14--26, 1992.

\end{thebibliography}
\bibliographystyle{alpha}
\end{document}